\documentclass[11pt]{article}

\usepackage{amssymb,amsmath,mathrsfs,amsthm,latexsym}
\usepackage[a4paper,includeheadfoot,left=15mm,right=15mm,top=15mm,bottom=25mm]{geometry} 
\usepackage{verbatim}
\usepackage{xcolor,graphicx}

\newtheorem{thm}{Theorem}[section]
\newtheorem{lem}[thm]{Lemma}
\newtheorem{prop}[thm]{Proposition}
\newtheorem{cor}[thm]{Corollary}

\theoremstyle{definition}

\theoremstyle{remark}
\newtheorem{rmk}[thm]{Remark}

\renewcommand{\Re}{\operatorname*{Re}}                             
\renewcommand{\Im}{\operatorname*{Im}}                             
\newcommand{\D}{\ensuremath{\,\mathrm{d}}}						   
\newcommand{\Mspacer}{\hspace{0.5mm}}                              
\newcommand{\M}[3]{#1_{#2\Mspacer#3}}                              
\renewcommand{\geq}{\geqslant}                                     
\renewcommand{\leq}{\leqslant}                                     
\renewcommand{\epsilon}{\varepsilon}							   
\newcommand{\BE}{\begin{equation}}                                 
\newcommand{\EE}{\end{equation}}                                   
\newcommand{\be}{\begin{equation}}                                 
\newcommand{\ee}{\end{equation}}                                   
\newcommand{\BES}{\begin{equation*}}                               
\newcommand{\EES}{\end{equation*}}                                 
\newcommand{\BP}{\begin{pmatrix}}                                  
\newcommand{\EP}{\end{pmatrix}}                                    
\newcommand{\N}{\mathbb{N}}                                        
\newcommand{\R}{\mathbb{R}}                                        
\newcommand{\C}{\mathbb{C}}                                        
\newcommand{\la}{\lambda}
\newcommand\re{{\rm e}}

\makeatletter
	\newcommand{\biggg}{\bBigg@{3}}
	\newcommand{\Biggg}{\bBigg@{4}}
	\newcommand{\bigggg}{\bBigg@{5}}
	\newcommand{\Bigggg}{\bBigg@{6}}
\makeatother

\newcommand{\superscript}[1]{\ensuremath{^{\textrm{#1}}}}

\newcommand{\Thns}[0]{\superscript{th}}
\newcommand{\Th}[0]{\Thns~}

\def\clap#1{\hbox to 0pt{\hss#1\hss}}

\numberwithin{equation}{section}

\title{The diffusion equation with nonlocal data}
\author{P. D. Miller$^1$ and  D. A. Smith$^{2}$\\
\footnotesize 1. Department of Mathematics, University of Michigan, Ann Arbor, MI, USA \\
\footnotesize 2. Yale-NUS College, Singapore\textup{: \texttt{dave.smith@yale-nus.edu.sg}}
}

\begin{document}
\maketitle

\abstract{
	We study the diffusion (or heat) equation on a finite 1-dimensional spatial domain, but we replace one of the boundary conditions with a ``nonlocal condition'', through which we specify a weighted average of the solution over the spatial interval.
	We provide conditions on the regularity of both the data and weight for the problem to admit a unique solution, and also provide a solution representation in terms of contour integrals.
	The solution and well-posedness results rely upon an extension of the Fokas (or unified) transform method to initial-nonlocal value problems for linear equations; the necessary extensions are described in detail.
	Despite arising naturally from the Fokas transform method, the uniqueness argument appears to be novel even for initial-boundary value problems.
}

\section{Introduction}

Consider the apparatus arranged as in figure~\ref{fig:Apparatus}.
\begin{figure}
	\begin{center}
		\includegraphics{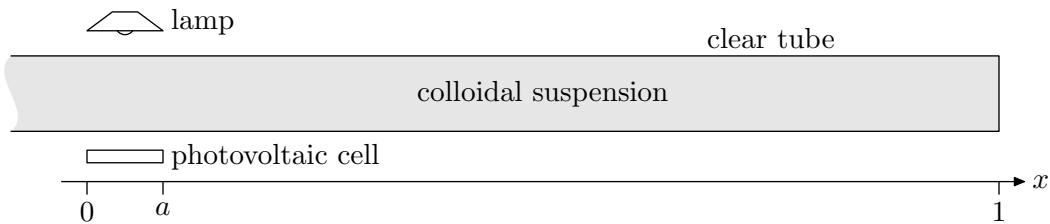}
		\caption{Measurement of concentration of dispersed substance in colloidal suspension}
		\label{fig:Apparatus}
	\end{center}
\end{figure}
A clear tube contains a colloidal suspension whose opacity is a known monotonic function of the concentration of the dispersed substance. At $x=1$, the tube is terminated so that the flux of dispersed substance across the boundary is zero. Assuming no net flow of the liquid phase, no variation in viscosity or temperature, and no external agitation, the dispersed substance diffuses according to the $1$-dimensional heat equation. Therefore, assuming the initial concentration profile, $q_0$, is known, a measurement, $\gamma$, of the concentration of dispersed substance at position $x=0$ for all time $t>0$ specifies the well-posed initial-boundary value problem
\begin{align*}
	[\partial_t-\partial_x^2]q(x,t) &= 0 & (x,t) &\in (0,1)\times(0,T), \\
	q(x,0) &= q_0(x) & x &\in[0,1], \\
	q_x(1,t) &= 0 & t &\in [0,T], \\
	q(0,t) &= \gamma(t) & t &\in [0,T],
\end{align*}
which may be solved, via a classical Fourier series or Green's function approach, for the concentration $q(x,t)$ at any interior point.

One must consider how a measurement of $\gamma$ could practically be made. One approach is to use a lamp (or laser) and photovoltaic cell to measure the opacity of the colloidal suspension, thereby to deduce the concentration of the dispersed substance. However, in any such apparatus, the photovoltaic cell must have some finite width $a\in(0,1)$, therefore the measurement will not be of $q(0,t)=\gamma(t)$, but rather of the average concentration
$$
	\frac{1}{a}\int_0^aq(x,t)\D x = \hat{\gamma}(t).
$$
Even a measurement of $\hat{\gamma}$ represents an idealized situation, in which the sensitivity of the photovoltaic cell is constant over its entire width $[0,a]$. More realistically, one may practically measure
$$
	\int_0^a K(x) q(x,t) \D x = g_0(t),
$$
for $K(x)$ a known nonnegative measure of the sensitivity of the photovoltaic cell, at position $x$.

Given a measurement of $g_0$, it is reasonable to ask whether one may deduce the concentration of the dispersed substance at any position within the colloidal suspension, for any positive time.
Specifically, one may ask whether the following initial-nonlocal value problem is well-posed, and how its solution may be determined:
\begin{subequations} \label{eqn:INVP}
\begin{align} \label{eqn:INVP.PDE}
	[\partial_t-\partial_x^2]q(x,t) &= 0 & (x,t) &\in (0,1)\times(0,T), \\ \label{eqn:INVP.IC}
	q(x,0) &= q_0(x) & x &\in[0,1], \\ \label{eqn:INVP.BC}
	q_x(1,t) &= g_1(t) & t &\in [0,T], \\ \label{eqn:INVP.NC}
	\int_0^1 K(x)q(x,t) \D x &= g_0(t) & t &\in [0,T];
\end{align}
we seek $q:[0,1]\times[0,T]$ for which
\begin{align}
	&x\mapsto q(x,\cdot)\mbox{ is a continuous map }(0,1)\to C^1[0,T], \\
	&t\mapsto q(\cdot,t)\mbox{ is a continuous map }(0,T)\to C^2[0,1].
\end{align}
\end{subequations}
In the above description it was assumed that $g_1=0$ and $K$ is supported on some small interval $[0,a]$, with $a\ll1$, but we find that the more general case introduces no mathematical complications.

Related problems have been studied, particularly by Cannon, since the 60's.
Deckert and Maple~\cite{DM1963a} establish existence and uniqueness for the simplification of problem~\eqref{eqn:INVP} with $K=1$. Cannon~\cite{Can1963a} improved this to allow $K$ constant on its support, but with support varying in time.
Numerical work followed~\cite{CEH1987a} (see references therein for similar problems).
In~\cite{CLM1993a}, existence, uniqueness, and numerical solution are studied for a 2-dimensional analogue.
However all of the above works assume that $K$ is constant on an interval $[0,a]$, and $0$ elsewhere. In~\cite{PS2015b}, this problem is solved for $K$ piecewise linear, by one of the present authors.

The present work has three purposes.
Firstly, we solve problem~\eqref{eqn:INVP} by proving Theorem~\ref{thm:Solution}.
Secondly, we describe the general extension of the Fokas transform method from problems with boundary conditions to problems with nonlocal conditions.
Thirdly, we give a uniqueness argument applicable to both initial-boundary value problems and initial-nonlocal value problems.

\begin{thm} \label{thm:Solution}
	Suppose
	\begin{enumerate}
		\item[(i)]{$q_0$, $g_0$, $g_1$ are differentiable functions, with bounded derivative.}
		\item[(ii)]{$K$ is a function of bounded variation, continuous at $0$, with $K(0)\neq0$.}
	\end{enumerate}
	Then
	problem~\eqref{eqn:INVP} has a unique solution.

	Moreover, there exists $R>0$ sufficiently large that, for all $\tau\in[t,T]$,
	\begin{subequations} \label{eqn:Solution}
	\begin{multline} \label{eqn:Solution.q}
		q(x,t) = \frac{1}{2\pi}\int_{-\infty}^{\infty}\re^{i\lambda x-\lambda^2t}\hat{q}_0(\lambda)\D\lambda - \frac{1}{2\pi}\int_{\partial D_R^+}\re^{i\lambda x-\lambda^2t}\frac{\zeta^+(\lambda;q_0) + H(\lambda;g_0,g_1,\tau)}{\Delta(\lambda)}\D\lambda \\
		- \frac{1}{2\pi}\int_{\partial D_R^-}\re^{i\lambda x-\lambda^2t}\frac{\re^{-i\lambda}\zeta^-(\lambda;q_0) + H(\lambda;g_0,g_1,\tau)}{\Delta(\lambda)}\D\lambda
	\end{multline}
	satisfies the initial-nonlocal value problem~\eqref{eqn:INVP}, where $\hat{q}_0$ is the Fourier transform of (the zero extension of) the initial datum,
	\begin{align} \label{eqn:Solution.Delta}
		\Delta(\lambda)  &= \int_0^1 K(y)\cos([1-y]\lambda)\D y, \\ \label{eqn:Solution.zeta+}
		\zeta^+(\lambda;q_0) &= \int_0^1K(y)\cos([1-y]\lambda)\int_0^y\re^{-i\lambda z}q_0(z)\D z\D y + \int_0^1 K(y) \re^{-i\lambda y} \int_y^1\cos([1-z]\lambda)q_0(z)\D z\D y, \\ \label{eqn:Solution.zeta-}
		\zeta^-(\lambda;q_0) &= i \int_0^1 K(y) \int_y^1\sin([z-y]\lambda)q_0(z)\D z\D y, \\ \label{eqn:Solution.H}
		H(\lambda;g_0,g_1,\tau) &= i\lambda\re^{-i\lambda} \int_0^\tau \re^{\lambda^2s}g_0(s)\D s + \int_0^1 K(y)\re^{-i\lambda y}\D y \int_0^\tau \re^{\lambda^2s}g_1(s)\D s, \\
		\intertext{and where the domains} \label{eqn:Solution.DRpm}
		D_R^\pm &= \{\lambda\in\C^\pm: \Re(\lambda^2)<0 \mbox{ and } |\lambda|>R\}
	\end{align}
	\end{subequations}
	have positively-oriented boundary.
	A sufficiently large $R$ is given in lemma~\ref{lem:Zeros}.
\end{thm}
The solution formula can be used to establish the following corollary, which gives much stronger regularity than problem~\eqref{eqn:INVP} requires.
\begin{cor} \label{cor:RegularityOfSolution}
	Suppose that $q(x,t)$ is given by equation~\eqref{eqn:Solution.q}.
	Then
	\begin{align}
		&x\mapsto q(x,\cdot)\mbox{ is a continuous map }(0,1)\to C^\infty(0,T), \\
		&t\mapsto q(\cdot,t)\mbox{ is a continuous map }(0,T)\to C^\infty[0,1].
	\end{align}
\end{cor}

\subsubsection*{Layout of paper}

In section~\ref{sec:FTM.Overview} we give an overview of the three stages of the Fokas transform method, as it has been applied to related problems. We also state the two principal lemmata upon which the Fokas transform method for initial-nonlocal value problem~\eqref{eqn:INVP} relies. The proofs of these lemmata are given in appendix~\ref{sec:AppProofs}.

In sections~\ref{sec:Unique} and~\ref{sec:Exist}, we implement stages~1--2, and~3 of the Fokas transform method, respectively.
Specifically, under the assumption of existence of a sufficiently smooth solution, the arguments of section~\ref{sec:Unique} establish the uniqueness and solution representation results of theorem~\ref{thm:Solution}; in section~\ref{sec:Exist} we show explicitly that the solution previously obtained satisfies initial-nonlocal value problem~\eqref{eqn:INVP}, thereby establishing the existence result of theorem~\ref{thm:Solution}.
The proof of corollary~\ref{cor:RegularityOfSolution} concludes section~\ref{sec:Exist}.

In section~\ref{sec:Multipoint}, we investigate relationships between initial-nonlocal value problem~\eqref{eqn:INVP} and certain initial-multipoint value problems of the form studied in~\cite{PS2015b}.
In section~\ref{sec:General}, we give the necessary extensions to the approach of the earlier sections required to study general initial-nonlocal value problems for the heat equation.

\section{Overview of the Fokas transform method} \label{sec:FTM.Overview}

The Fokas transform method~\cite{Fok2008a} (or unified transform method) has been used to solve a variety of initial-boundary value problems for linear and nonlinear evolution partial differential equations; see below for a brief overview of the method, and~\cite{DTV2014a} for an extended pedagogical introduction.
Finite interval problems for simple boundary conditions were solved by Fokas and Pelloni~\cite{FP2001a,Pel2004a} and later generalized to more complicated boundary conditions~\cite{Smi2012a,Smi2013a}.
The method was extended, largely by Sheils, to interface problems for second order equations~\cite{APSS2015a,DPS2014a,DS2014a,DS2015a,SS2015a} and higher order equations~\cite{DSS2015a}.
The Sheils formulation of the Dirichlet-to-Neumann map for interface problems was adapted to study multipoint problems~\cite{PS2015b}.
Via reduction to multipoint problems, the latter paper also enables study of second order nonlocal problems for continuous piecewise linear weights $K$.

The Fokas transform method, as applied to two-point initial-boundary value problems for linear evolution equations of the form
\BE \label{eqn:GeneralPDE}
	\left[\partial_t+\omega(-i\partial_x)\right]q(x,t)=0,
\EE
for degree $n$ polynomial $\omega$,
may be understood as a 3-stage method in the following way.
\begin{description}
	\item[Stage 1]{
		Assuming existence of a sufficiently smooth function $q$ satisfying both the $n$\Th order partial differential equation and the initial condition, obtain a pair of equations that $q$ must satisfy. Specifically, one derives the \emph{global relation}, a linear equation which relates $2n+2$ quantities: the Fourier transform of $q(\cdot,t)$, the Fourier transform of the initial datum, and certain time-transforms of the boundary values of $q$ and its first $n-1$ spatial derivatives. These time-transformed boundary values are known as \emph{spectral functions}.
		The second important equality, often known as the \emph{Ehrenpreis form}, provides a representation of $q$ in terms of certain contour integrals of the Fourier transform of the initial datum and the same $2n$ spectral functions that appear in the global relation.
	}
	\item[Stage 2]{
		For a well-posed initial-boundary value problem, $n$ of the boundary values (for example, with $n=2$, the Dirichlet data at $x=0$ and $x=1$) may be specified by the boundary conditions, but the other $n$ (continuing the example, Neumann) boundary values are not specified in the problem. Therefore, in order to use the Ehrenpreis form as an effective integral representation, it is necessary to obtain formulae for each of the unknown spectral functions in terms of the data of the problem. This process of constructing spectral functions from data is known as the \emph{Dirichlet-to-Neumann map}, even if the unknown boundary values are not actually $q_x(0,t)$ and $q_x(1,t)$.

		In general, the Dirichlet-to-Neumann map may be expressed as a linear system built from the boundary conditions and the global relation.
		The ratios $\zeta^\pm/\Delta$ are the ratios of determinants that appear when Cramer's rule is used to solve the linear system.
		It is established, via a contour deformation argument, that resulting terms involving the Fourier transform of $q(\cdot,t)$ do not contribute to the solution representation.

		At the end of stage 2, one has shown that any solution of the problem is necessarily given in terms of the data by a specific contour integral representation.
		It therefore only remains to show that this integral representation actually solves the problem, i.e.\ to establish existence of a solution.
		Uniqueness of the solution is therefore a direct consequence of the method of construction of the integral representation.
		It appears that this uniqueness argument via the Fokas transform method has not previously been explicitly described, even for initial-boundary value problems.
	}
	\item[Stage 3]{
		Defining $q$ by the formula obtained in stage 2, we show directly that $q$ satisfies the initial-boundary value problem, thereby establishing existence of a solution.
		It turns out that the contour integrals used to define $q$ converge uniformly in $(x,t)$ up to the boundaries, which simplifies the process of evaluating $q$ on the boundaries.
		The space and time dependence of $q$ is very simple, so it is typically easy to see that $q$ satisfies the PDE.
	}
\end{description}

Generalising the Fokas method from initial-boundary value problems to initial-nonlocal value problems requires new implementations of stages~2 and~3.
The implementation of stages~2 and~3 requires repeated use of two technical lemmata, one concerning the locus of zeros of the determinant $\Delta$, and the other concerning the boundedness and decay of determinant ratios $\zeta^\pm/\Delta$ in $\overline{D_R^\pm}$.
For the initial-nonlocal problem~\eqref{eqn:INVP}, these lemmata take the following forms.

\begin{lem} \label{lem:Zeros}
	Suppose $k(y)=K(1-y)$ has support $[a,b]\subseteq[0,1]$, has bounded total variation $V_{0}^{1}(k)<\infty$, is left-continuous at the endpoint $b$, and $k(b)\neq0$.
	Choose $\delta_0\in(a,b)$ such that $V_{b-\delta_0}^b(k)<|k(b)|/8$.
	Define $\Delta$ as in theorem~\ref{thm:Solution}.
	Then the zeros of $\Delta(\lambda)$ all have imaginary part less than
	\BES
		M=\max\left\{ \frac{\log 2}{b} , \frac{1}{\delta_0}\log\left( \frac{4V_{0}^{1}(k)}{|k(b)|} \right) \right\}.
	\EES
	Moreover, choosing $R=\sqrt{2}M$, all zeros of $\Delta$ lie exterior to $D_R^\pm$.
\end{lem}
\noindent
No attempt has been made to optimize $R$.

\begin{lem} \label{lem:Decay}
	Suppose $k(y)=K(1-y)$ has support $[a,b]\subseteq[0,1]$, has bounded total variation $V_{0}^{1}(k)<\infty$, is left-continuous at the endpoint $b$, and $k(b)\neq0$.
	Suppose also $\|\phi'\|_{\infty}<\infty$ and define $\zeta^\pm$, $\Delta$ as in theorem~\ref{thm:Solution}.
	Then, as $\lambda\to\infty$ from within $\overline{D_R^-}$,
	\BE \label{eqn:DecayLem.-}
		\frac{\zeta^-(\lambda;\phi)}{\Delta(\lambda)} = \mathcal{O}(|\lambda|^{-1}),
	\EE
	uniformly in $\arg(\lambda)$.
	If, in addition, $b=1$ then, as $\lambda\to\infty$ from within $\overline{D_R^+}$,
	\BE \label{eqn:DecayLem.+}
		\frac{\zeta^+(\lambda;\phi)}{\Delta(\lambda)} = \mathcal{O}(1),
	\EE
	uniformly in $\arg(\lambda)$.
\end{lem}

Proofs of lemmata~\ref{lem:Zeros} and~\ref{lem:Decay} are given in appendix~\ref{sec:AppProofs}.

\section{Uniqueness} \label{sec:Unique}

\subsection{Fokas transform method for heat equation on $[0,1]$. Global relation and Ehrenpreis form (stage 1)} \label{ssec:FTM.FiniteHeat}

Suppose $q:[0,1]\times[0,T]\to\C$ satisfies the partial differential equation~\eqref{eqn:INVP.PDE}, initial condition~\eqref{eqn:INVP.IC}, and the regularity conditions of theorem~\ref{thm:Solution}. For $(x,t)\in[0,1]\times[0,T]$ and a spectral parameter $\lambda\in\C$, we define the functions
\begin{align*}
	X(x,t,\lambda) &= \re^{-i\lambda x+\lambda^2t}q(x,t) \\
	Y(x,t,\lambda) &= \re^{-i\lambda x+\lambda^2t}\left[ \partial_x + i\lambda \right]q(x,t),
\end{align*}
which have the property $\partial_tX-\partial_xY=0$ for all $\lambda\in\C$. Applying Green's theorem to the space-time rectangle $\Omega=(y,z)\times(0,\tau)$, we obtain
\BE \label{eqn:GreensThm}
	0=\int_{\Omega} (\partial_tX-\partial_xY) \D x\D t = \int_{\partial\Omega} (Y\D t + X\D x).
\EE
Using the notation
\begin{align*}
	\hat{q}_0(\lambda;y,z) &= \int_y^z\re^{-i\lambda\xi}q_0(\xi)\D\xi \\
	\hat{q}(\lambda,\tau;y,z) &= \int_y^z\re^{-i\lambda\xi}q(\xi,\tau)\D\xi \\
	f_0(\lambda;y,\tau) &= i\lambda\int_0^\tau\re^{\lambda^2s}q(y,s)\D s \\
	f_1(\lambda;y,\tau) &= \int_0^\tau\re^{\lambda^2s}q_x(y,s)\D s
\end{align*}
equation~\eqref{eqn:GreensThm} implies, for all $y,z$ satisfying $0\leq y \leq z \leq 1$, all $\tau\in[0,T]$ and all $\lambda\in\C$, the \emph{global relation}
\BE \label{eqn:GR}
	\hat{q}_0(\lambda;y,z) - \re^{\lambda^2\tau}\hat{q}(\lambda,\tau;y,z) = \re^{-i\lambda y}\left[f_0(\lambda;y,\tau) + f_1(\lambda;y,\tau)\right] - \re^{-i\lambda z}\left[f_0(\lambda;z,\tau) + f_1(\lambda;z,\tau)\right].
\EE
Typically, for an initial-boundary value problem, it is only necessary to have the global relation for $y=0$, $z=1$. It is clear that $\hat{q}_0(\lambda;0,1)$ is the ordinary Fourier transform of the initial datum $q_0$ extended to $\R$ by the zero function. Similarly, $\hat{q}(\lambda,t;0,1)$ is the ordinary spatial Fourier transform of the solution $q$ at time $t$. The spectral functions $f_0$ and $f_1$ are time-transforms of the solution and its derivative evaluated at position $y$; so if $y\in\{0,1\}$ then these are time-transforms of the boundary values. However, it will be essential to use this more general form of the global relation, in which $y\geq0$ and $z\leq1$, to implement stage 2 of the Fokas transform method for our nonlocal problem.

Evaluating the global relation at $\tau=t$, $y=0$, $z=1$, making $\hat{q}(\lambda,t;0,1)$ the subject of the equation and applying the inverse Fourier transform, we obtain the formula
\begin{multline} \label{eqn:Ehrenpreis.0}
	2\pi q(x,t) = \int_{-\infty}^\infty \re^{i\lambda x - \lambda^2t} \hat{q}_0(\lambda;0,1)\D\lambda - \int_{-\infty}^\infty \re^{i\lambda x - \lambda^2t}\left[f_0(\lambda;0,t) + f_1(\lambda;0,t)\right]\D\lambda \\
	+ \int_{-\infty}^\infty \re^{i\lambda (x-1) - \lambda^2t}\left[f_0(\lambda;1,t) + f_1(\lambda;1,t)\right]\D\lambda.
\end{multline}
We define the sectors
\BES
	D^\pm = \{\lambda\in\C^\pm:\Re(\lambda^2)<0\}
\EES
in the upper and lower halves of the complex $\lambda$-plane, and orient their boundaries $\partial D^\pm$ in the positive sense. Integrating by parts,
$$
	\re^{-\lambda^2t}\left[f_0(\lambda;0,t) + f_1(\lambda;0,t)\right] = \mathcal{O}\left(|\lambda|^{-1}\right),
$$
uniformly in $\arg(\lambda)$, as $\lambda\to\infty$ within $\lambda\in\overline{\C^+\setminus D^+}$. Hence, by Jordan's lemma, for $x\in(0,1]$,
$$
	\int_{\partial(\C^+\setminus D^+)}\re^{i\lambda x - \lambda^2t}\left[f_0(\lambda;0,t) + f_1(\lambda;0,t)\right]\D\lambda=0.
$$
This allows a contour deformation in the second integral of~\eqref{eqn:Ehrenpreis.0} from $\R$ to $\partial D^+$.
A similar argument allows a contour deformation from $-\R$ to $\partial D^-$ in the third integral.
Moreover, by Cauchy's theorem, as the integrands are all entire functions, we may deform the contours of integration over any bounded region.
We deform $\partial D^\pm$ to $\partial D_R^\pm$, where $R$ is as specified by lemma~\ref{lem:Zeros}. This establishes the Ehrenpreis form:
\begin{multline} \label{eqn:Ehrenpreis}
	2\pi q(x,t) = \int_{-\infty}^\infty \re^{i\lambda x - \lambda^2t} \hat{q}_0(\lambda;0,1)\D\lambda - \int_{\partial D_R^+} \re^{i\lambda x - \lambda^2t}\left[f_0(\lambda;0,t) + f_1(\lambda;0,t)\right]\D\lambda \\
	- \int_{\partial D_R^-} \re^{i\lambda (x-1) - \lambda^2t}\left[f_0(\lambda;1,t) + f_1(\lambda;1,t)\right]\D\lambda.
\end{multline}

\subsection{Nonlocal Dirichlet-to-Neumann map (stage 2)} \label{ssec:NonlocalDtoNMap}

To implement a Dirichlet-to-Neumann map, we must find expressions for each of the four boundary spectral functions which appear in Ehrenpreis form~\eqref{eqn:Ehrenpreis}:
\begin{align*}
	&f_0(\lambda;0), & &f_1(\lambda;0), & &f_0(\lambda;1), & &f_1(\lambda;1);
\end{align*}
the explicit $t$-dependence has been dropped to simplify the notation.
Note that, for $q$ satisfying the boundary condition~\eqref{eqn:INVP.BC},
\BE \label{eqn:h1.defn}
	f_1(\lambda;1) = \int_0^t \re^{\lambda^2s}g_1(s)\D s =: h_1(\lambda)
\EE
is known explicitly in terms of the boundary datum $g_1$.
We now assume that $q$ satisfies both the boundary condition and the nonlocal condition~\eqref{eqn:INVP.NC} and construct a linear system which may be solved for the remaining three spectral functions.

Evaluating the global relation~\eqref{eqn:GR} at $y=0$, $z=1$, $\tau=t$ we obtain the \emph{global relation of two-point type}
\BE \label{eqn:GR.usual}
	f_0(\lambda;0) + f_1(\lambda;0) = \re^{-i\lambda}f_0(\lambda;1) + \re^{-i\lambda}h_1(\lambda) + \hat{q}_0(\lambda;0,1) - \re^{\lambda^2t}\hat{q}(\lambda,t;0,1).
\EE
The global relation of two-point type relates the three unknown quantities with data of the problem and also the Fourier transform of $q$ at time $t$. We proceed, treating the latter as if it were a datum, and show afterwards that it does not contribute to our solution representation.

When applying the Fokas transform method to an initial-boundary value problem for the heat equation, a second boundary condition would be available, specifying some linear combination of the boundary spectral functions, which could be used to reduce the global relation of two-point type to an equation in only two unknowns.
For our problem, we have nonlocal condition~\eqref{eqn:INVP.NC} in place of a second boundary condition.
Applying the time transform to the nonlocal condition, we obtain
\BES
	\int_0^1 K(y) f_0(\lambda;y)\D y = i\lambda \int_0^t \re^{\lambda^2s} \int_0^1 K(y) q(y,s) \D y \D s = i\lambda \int_0^t \re^{\lambda^2s} g_0(s) \D s =: h_0(\lambda).
\EES
The new datum $h_0$ does not appear in equation~\eqref{eqn:GR.usual}, so it cannot be used to reduce that equation.
However it is essential that we use the nonlocal condition in some way in constructing the Dirichlet-to-Neumann map; otherwise, it would be possible to implement a Dirichlet-to-Neumann map with only the one boundary condition~\eqref{eqn:INVP.BC}, thence to solve an underspecified initial-boundary value problem.
We apply another evaluation operation to the global relation~\eqref{eqn:GR} to obtain an equation linking $h_0$ to one of our unknowns. Specifically, we evaluate equation~\eqref{eqn:GR} at $z=1$, multiply by $\re^{i\lambda y}K(y)$ and integrate from $0$ to $1$ in $y$:
\begin{multline} \label{eqn:GR.nonlocal}
	\int_0^1 K(y) f_1(\lambda;y) \D y - f_0(\lambda;1) \re^{-i\lambda} \int_0^1 \re^{i\lambda y} K(y) \D y =
	-h_0(\lambda) + h_1(\lambda)\re^{-i\lambda}\int_0^1 \re^{i\lambda y} K(y) \D y +{} \\
	\int_0^1 \re^{i\lambda y} K(y)\hat{q}_0(\lambda;y,1)\D y - \re^{\lambda^2 t}\int_0^1 \re^{i\lambda y} K(y)\hat{q}(\lambda,t;y,1)\D y.
\end{multline}
This \emph{global relation of nonlocal type} is a linear equation in the spectral function $f_0(\lambda;1)$ and the new unknown $\int_0^1 K(y) f_1(\lambda;y) \D y$.

Observe that the spectral functions and spectral data have particularly simple symmetry properties:
\begin{align*}
	f_0(-\lambda;0) &= -f_0(\lambda;0), & f_0(-\lambda;1) &= -f_0(\lambda;1), & h_0(-\lambda) &= -h_0(\lambda), \\
	f_1(-\lambda;0) &= f_1(\lambda;0), & h_1(-\lambda) &= h_1(\lambda) & f_1(-\lambda;y) &= f_1(\lambda;y).
\end{align*}
This means that by applying the maps $\lambda\mapsto\lambda$ and $\lambda\mapsto-\lambda$ to the global relation of nonlocal type, we obtain a system of two linear equations for $f_0(\lambda;1)$ and $\int_0^1 K(y) f_1(\lambda;y) \D y$. Having solved this system for $f_0(\lambda;1)$, the global relation of two-point type yields an expression for $f_0(\lambda;0) + f_1(\lambda;0)$. It is possible to apply the maps $\lambda\mapsto\lambda$ and $\lambda\mapsto-\lambda$ to the global relation of two-point type, hence solve for each of these spectral functions separately, but, as it is this sum of spectral functions that appears in the Ehrenpreis form, it is unnecessary to do so.

Explicitly, substituting into the Ehrenpreis form~\eqref{eqn:Ehrenpreis}, this yields
\begin{multline} \label{eqn:Solution.with.qhat}
	2\pi q(x,t) = \int_{-\infty}^{\infty}\re^{i\lambda x-\lambda^2t}\hat{q}_0(\lambda)\D\lambda \\
	- \int_{\partial D_R^+}\re^{i\lambda x-\lambda^2t}\frac{\zeta^+(\lambda;q_0) + H(\lambda;g_0,g_1,t)}{\Delta(\lambda)}\D\lambda
	- \int_{\partial D_R^-}\re^{i\lambda x-\lambda^2t}\frac{\re^{-i\lambda}\zeta^-(\lambda;q_0) + H(\lambda;g_0,g_1,t)}{\Delta(\lambda)}\D\lambda \\
	+ \int_{\partial D_R^+}\re^{i\lambda x}\frac{\zeta^+(\lambda;q(\cdot,t))}{\Delta(\lambda)}\D\lambda
	+ \int_{\partial D_R^-}\re^{i\lambda(x-1)}\frac{\zeta^-(\lambda;q(\cdot,t))}{\Delta(\lambda)}\D\lambda,
\end{multline}
where $\zeta^\pm$, $\Delta$, and $H$ are as defined in theorem~\ref{thm:Solution}.

Note that we have split each of the second and third integrands of the Ehrenpreis form into two parts. It should be justified that each of these yields a convergent integral. Observe that each integrand is a meromorphic function, whose only poles are zeros of the denominator $\Delta$. By lemma~\ref{lem:Zeros}, all zeros of $\Delta$ lie to the right of $\partial D_R^\pm$. We will use lemma~\ref{lem:Decay} to show that the fourth and fifth integrals of equation~\eqref{eqn:Solution.with.qhat} each evaluate to $0$, which also justifies convergence of the second and third integrals.

It is immediate from Jordan's lemma and lemmata~\ref{lem:Decay} and~\ref{lem:Zeros} that the fifth integral of equation~\eqref{eqn:Solution.with.qhat} evaluates to zero for all $x\in[0,1)$. For the fourth integral, observe that by lemma~\ref{lem:Decay}
\BE \label{eqn:Awkward.Decay}
	\re^{i\lambda x/2}\frac{\zeta^+(\lambda;\phi)}{\Delta(\lambda)} = \mathcal{O}\left( \re^{-x|\lambda|/2\sqrt{2}} \right)
\EE
uniformly in $\arg(\lambda)$ as $\lambda\to\infty$ from within $\overline{D_R^+}$, provided $x\in(0,1]$ and $K(0)\neq0$. Hence, by Jordan's lemma and lemma~\ref{lem:Zeros}, the fourth integral of equation~\eqref{eqn:Solution.with.qhat} evaluates to zero for all $x\in(0,1]$.

Finally, we argue that $\tau>t$ may replace $t$ in the fourth argument of $H$, by showing that
\BE \label{eqn:H.asymp}
	\int_{\partial D_R^\pm} \re^{i\lambda x-\lambda^2t}\frac{H(\lambda;g_0,g_1,\tau)-H(\lambda;g_0,g_1,t)}{\Delta(\lambda)}\D\lambda = 0.
\EE
By definition,
$$
	\re^{i\lambda}\left[H(\lambda;g_0,g_1,\tau)-H(\lambda;g_0,g_1,t)\right] = i\lambda \int_t^\tau \re^{\lambda^2s}g_0(s)\D s + \int_0^1 K(y)\re^{i\lambda(1-y)}\D y \int_t^\tau \re^{\lambda^2s}g_1(s)\D s.
$$
Integration by parts in $s$ and equation~\eqref{eqn:Delta.asymp} establishes that, as $\lambda\to\infty$ from within $\overline{D_R^-}$,
$$
	\frac{\re^{i\lambda}\left[H(\lambda;g_0,g_1,\tau)-H(\lambda;g_0,g_1,t)\right]}{\Delta(\lambda)} = \mathcal{O}(|\lambda|^{-2}),
$$
uniformly in $\arg(\lambda)$. Hence, by Jordan's lemma, the $\partial D_R^-$ version of equation~\eqref{eqn:H.asymp} holds. Similarly, as $\lambda\to\infty$ from within $\overline{D_R^+}$,
$$
	\frac{\left[H(\lambda;g_0,g_1,\tau)-H(\lambda;g_0,g_1,t)\right]}{\Delta(\lambda)} = \mathcal{O}(1),
$$
uniformly in $\arg(\lambda)$. Once again, we exploit the exponential decay of $\re^{i\lambda x/2}$ as $\lambda\to\infty$ from within $\overline{D_R^+}$ and Jordan's lemma to conclude that the $\partial D_R^+$ version of equation~\eqref{eqn:H.asymp} holds.

We have established, under the assumption of existence of a solution $q$ with appropriate regularity, that the solution is unique and may be represented by equation~\eqref{eqn:Solution.q}.

\begin{rmk} \label{rmk:AltGR}
	We used the global relation of two-point type~\eqref{eqn:GR.usual} and the global relation of nonlocal type~\eqref{eqn:GR.nonlocal}, under the maps $\lambda\mapsto\pm\lambda$, to obtain a system of equations solvable for the spectral functions. It is possible to obtain a \emph{second global relation of nonlocal type},
	\begin{multline*}
		\left[ f_0(\lambda;0) + f_1(\lambda;0) \right] \int_0^1 \re^{i\lambda z} K(z) \D z - \int_0^1 K(z) f_1(\lambda;z)\D z \\
		= h_0(\lambda)
		+ \int_0^1 \re^{i\lambda y} K(z)\hat{q}_0(\lambda;0,z)\D z - \re^{\lambda^2 t}\int_0^1 \re^{i\lambda y} K(z)\hat{q}(\lambda,t;0,z)\D z,
	\end{multline*}
	by evaluating the global relation~\eqref{eqn:GR} at $y=0$, multiplying by $\re^{i\lambda z}K(z)$ and integrating from $0$ to $1$ in $z$.
	Any two of the global relation of two-point type, the global relation of nonlocal type, and the second global relation of nonlocal type, under the maps $\lambda\mapsto\pm\lambda$, provide a system of equations solvable for the spectral functions, providing somewhat different solution representations.
	It is not clear that there is any advantage in choosing any particular system, except that the one chosen above yields a simpler expression.
\end{rmk}

\begin{rmk} \label{rmk:SidewaysHeat}
	A crucial component of this argument is the requirement of lemma~\ref{lem:Decay} that $0$ be in the support of $K$.
	If instead the support of $K$ is $[y,z]$ for some $y>0$, then the Jordan's lemma argument can only be made for $x\in(y,1]$, and our solution representation is only valid for such $x$.
	Solving the problem for all $x\in(y,1]$, we could then set up the following initial-boundary value problem for $u$:
	\begin{align*}
		[\partial_t-\partial_x^2]u(x,t) &= 0 & (x,t) &\in (0,y)\times(0,T), \\
		u(x,0) &= q_0(x) & x &\in [0,y], \\
		u(y,t) &= q(y,t) & t &\in [0,T], \\
		u_x(y,t) &= q_x(y,t) & t &\in [0,T].
	\end{align*}
	However, this represents a sideways problem for the heat equation, which is ill-posed~\cite{Ber1999a}.
	Therefore, lacking any boundary / nonlocal data for $q$ to the left of $y$, it is unsurprising that it is impossible to directly solve for $q$ to the left of $y$.
\end{rmk}

\begin{rmk} \label{rmk:K0at0}
	Supposing the enhanced regularity criterion $K\in C^2[0,1]$, we can explicitly derive the leading order term in the asymptotic expansion of $\zeta^+(\la;\phi)/\Delta$, via a simple integration by parts argument.
	Indeed, as $\la\to\infty$ from within $\overline{D^+}$,
	\BE \label{eqn:DecayWithHigherKRegularity.+}
		\frac{\zeta^+(\lambda;\phi)}{\Delta(\lambda)} = \frac{-1}{K(0)}\int_0^1K(y)\phi(y)\D y + \mathcal{O}(|\lambda|^{-1}),
	\EE
	uniformly in $\arg(\lambda)$.
	Therefore, the $\mathcal{O}(1)$ result in equation~\eqref{eqn:DecayLem.+} is optimal in the sense that ${o}(1)$ is generally false.

	The Jordan's lemma argument in $D_R^+$ requires that one exploit $\re^{i\lambda x} = \re^{i\lambda x/2}\re^{i\lambda x/2}$ and use the fact that, when $\lambda\to\infty$ from within $\overline{D_R^+}$, necessarily $\Im(\lambda)\to\infty$ to see that one of these factors gives exponential decay in $|\lambda|$, while the other factor plays the role of the exponential kernel for Jordan's lemma.
	This contrasts with the Jordan's lemma argument in $D_R^-$, which did not require such an approach, because lemma~\ref{lem:Decay} already provides decay of $\zeta^-(\la;\phi)/\Delta(\la)$ in $D_R^-$.
	For initial-boundary value problems studied via Fokas method, this ``splitting $\re^{i\lambda x}$'' argument is always available for the heat equation, but is not available when studying problems of odd spatial order, or problems for the linear Schr\"{o}dinger equation.
	However, for initial-boundary value problems, it is never necessary to use the ``splitting $\re^{i\lambda x}$'' argument, as the result corresponding to lemma~\ref{lem:Decay} always takes the form
	``as $\lambda\to\infty$ from within $\overline{D_R^\pm}$,
	$\zeta^\pm(\lambda;\phi)/\Delta(\lambda) = \mathcal{O}(|\lambda|^{-1})$,
	uniformly in $\arg(\lambda)$'',
	for $\zeta^\pm$, $\Delta$, $D_R^\pm$ defined appropriately for the problem at hand.

	It is an open question whether a generalisation of the non-decaying asymptotic formula in lemma~\ref{lem:Decay} will present a serious obstruction to generalisation of the Fokas method to initial-nonlocal value problems for other equations, but it is expected that adding $\delta_0(x)$ to $K$ would circumvent this issue.
\end{rmk}

\begin{rmk} \label{rmk:Spectral}
	The solution representation obtained via the Fokas transform method takes the form of contour integrals.
	The Fokas method has been used to solve initial-boundary value problems for which classical Fourier series solution representations do not exist, and to determine well-posedness criteria for problems of high spatial order~\cite{Pel2004a,Smi2012a,Smi2013a}.
	The meaning of the solution representation in terms of the spectral theory of the spatial two-point differential operator has been studied~\cite{Pel2005a,PS2013a,Smi2015a,FS2016a,PS2016a}.
	It was shown that the integrals in the solution representation are a new species of spectral object.
	The spectral theory of the nonlocal heat operator is now open to the same analysis as the two-point operators, but such analysis is beyond the scope of the present work.
\end{rmk}

\begin{rmk} \label{rmk:Nonlocal.Boundary.Map}
	In order to obtain a formula for the Dirichlet boundary value $\gamma(t):=q(0,t)$, one may simply evaluate expression~\eqref{eqn:Solution} at $x=0$. In the case of initial-interface value problems, it is known that one may obtain a formula for (an appropriate equivalent of) $\gamma(t)$ without use of the Ehrenpreis form~\cite{DS2016a}. We argue that such an approach cannot be applied to our problem.

	By the validity of the usual inverse Fourier transform, for $\tau>t$,
	\BES
		\gamma(t) = q(0,t) = \frac{1}{2\pi}\int_{-\infty}^\infty \re^{i\rho t} \int_0^\tau \re^{-i\rho s} q(0,s) \D s \D\rho.
	\EES
	Applying the change of variables $\lambda^2=-i\rho$, $\lambda=i\sqrt{i\rho}$ (for the principal branch of the square root), we obtain a formula for $\gamma(t)$ in terms of $f_0(\lambda;0)$:
	\BE \label{eqn:gamma.t}
		\gamma(t) = \frac{1}{\pi} \int_{\partial D_R^+} \re^{-\lambda^2 t} f_0(\lambda;0) \D\lambda.
	\EE
	Solving the linear system described above, we find
	\begin{multline*}
		2f_0(\lambda;0) = \frac{\zeta^+(\lambda;q_0) - \zeta^+(-\lambda;q_0)}{\Delta(\lambda)} + \frac{H(\lambda;g_0,g_1,\tau) - H(-\lambda;g_0,g_1,\tau)}{\Delta(\lambda)} + \frac{\zeta^+(\lambda;q(\cdot;\tau)) - \zeta^+(-\lambda;q(\cdot;\tau))}{\Delta(\lambda)},
	\end{multline*}
	which can be substituted into equation~\eqref{eqn:gamma.t} to obtain an implicit expression for $\gamma(t)$.
	However, because $\zeta^+/\Delta$ does not decay as $\lambda\to\infty$ (see lemma~\ref{lem:Decay} and the first paragraph of remark~\ref{rmk:K0at0}), it is not possible to remove the effects of $q(\cdot;\tau)$ from the representation. Therefore, it is not possible to obtain an effective representation of $\gamma(t)$ in this way.
\end{rmk}

\begin{rmk} \label{rmk:Numerics}
	A full numerical implementation of the Fokas transform method for initial-nonlocal value problems is beyond the scope of this work.
	However, we provide a plot of $q(x,t)$ for a particular value of $K$ in figure~\ref{fig:plot.q}.
	The code used to produce this plot may be found at~\cite{Smi2017a}.
	\begin{figure}
		\centering
		\includegraphics[width=0.3\textwidth]{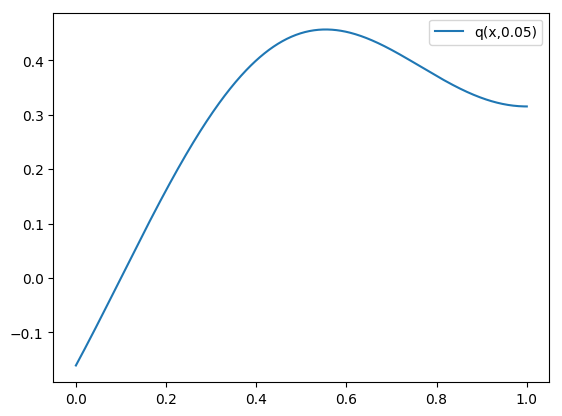}
		\includegraphics[width=0.4\textwidth]{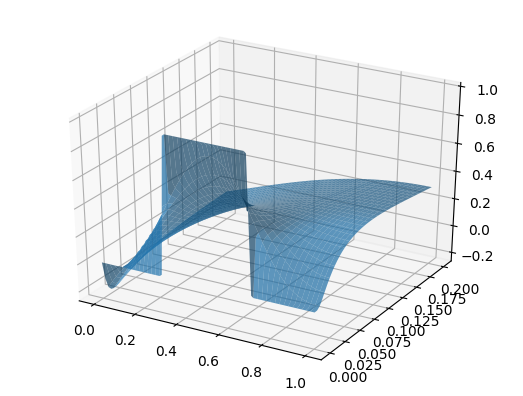}
		\caption{
			Solution of the homogeneous problem $t=0.05$, with box initial datum, $K=1$ on $(0,0.2)$ and $K=0$ elsewhere.
			The left figure is at time $t=0.05$.
		}
		\label{fig:plot.q}
	\end{figure}
	As figure~\ref{fig:plot.q} demonstrates, it is possible to choose $K$ such that, even with $g_0=g_1=0$, positive initial data yields a solution which is, for some $x$, negative at positive time.
	Of course, this would not occur if the homogeneous nonlocal condition was replaced with a homogeneous Dirichlet condition, but perhaps it is not entirely surprising.
	Indeed, viewing the nonlocal condition as a ``smeared out'' Dirichlet condition, we might understand that the smearing extends both out of and into the interval.
	Therefore, the approximated Dirichlet condition is not at $x=0$ but at $x\approx 0.1$, and $q(0.1,t)\approx0$ as expected.
\end{rmk}

\section{Existence (stage 3)} \label{sec:Exist}

We show that $q$ as defined in theorem~\ref{thm:Solution} solves the initial-nonlocal problem~\eqref{eqn:INVP}. Exploiting the same argument as was used at the end of stage~2, we note that the formula for $q$ is independent of choice of $\tau\in[t,T]$.

Defining
$$
	\Omega^\epsilon = \{(x,t)\in[0,1]\times[0,T]: \|(x,t)-(0,0)\|\geq\epsilon \mbox{ and } \|(x,t)-(1,0)\|\geq\epsilon]\},
$$
it is clear that, for all $\epsilon>0$, the integrals in equation~\eqref{eqn:Solution.q} converge uniformly in $(x,t)$ on $\Omega^\epsilon$. Moreover, all partial derivatives of $q$ exist on the interior of $\Omega^\epsilon$ and are given by taking the (uniformly convergent on any closed subset of the interior of $\Omega^\epsilon$) integrals of corresponding partial derivatives of the integrands.

Using $\tau>t$ in expression~\eqref{eqn:Solution.q}, the $x$ and $t$ dependence of $q$ is contained within the exponential kernel of each term. It follows immediately that $q$ satisfies the PDE~\eqref{eqn:INVP.PDE}.

Suppose $t=0$ and $x\in(0,1)$. Evaluating expression~\eqref{eqn:Solution.q} with $\tau=0$, we find $H=0$. By lemma~\ref{lem:Decay}, the second integrand in expression~\eqref{eqn:Solution.q} is $\re^{i\lambda x/2}\mathcal{O}(\re^{-|\lambda|/\sqrt{2}})$, and the third integrand is $\re^{i\lambda(x-1)}\mathcal{O}(|\lambda|^{-1})$ as $\lambda\to\infty$ from within $\overline{D_R^+}$ and within $\overline{D_R^-}$, respectively. Hence, by Jordan's lemma and lemma~\ref{lem:Zeros}, the second and third integrals of expression~\eqref{eqn:Solution.q} evaluate to $0$. The usual Fourier inversion theorem implies that $q$ satisfies the initial condition.

Suppose $t\in(0,T)$ and $\tau>t$.
By the uniform convergence of each integral in expression~\eqref{eqn:Solution.q},
\begin{multline} \label{eqn:Existence.Proof:qx1}
	q_x(1,t) = \frac{i}{2\pi} \lim_{x\to1^-} \left[ \int_{-\infty}^{\infty}\lambda\re^{i\lambda x-\lambda^2t}\hat{q}_0(\lambda)\D\lambda - \int_{\partial D_R^+}\lambda\re^{i\lambda x-\lambda^2t}\frac{\zeta^+(\lambda;q_0) + H(\lambda;g_0,g_1,\tau)}{\Delta(\lambda)}\D\lambda \right. \\
	\left. - \int_{\partial D_R^-}\lambda\re^{i\lambda x-\lambda^2t}\frac{\re^{-i\lambda}\zeta^-(\lambda;q_0) + H(\lambda;g_0,g_1,\tau)}{\Delta(\lambda)}\D\lambda \right].
\end{multline}

It is straightforward from the definitions of $\zeta^\pm$ and $\Delta$ to show that, for all $\lambda\in\C$,
\BE \label{eqn:zeta.cancel}
	\zeta^+(\lambda;\phi)-\re^{-i\lambda}\zeta^-(\lambda;\phi) = \Delta(\lambda) \int_0^1\re^{-i\lambda z}\phi(z)\D z.
\EE
Formally, it appears that one may apply equation~\eqref{eqn:zeta.cancel} to rewrite $\zeta^+$ in terms of $\zeta^-$ in equation~\eqref{eqn:Existence.Proof:qx1} and observe some cancellation. However certain resulting integrals would not converge. Instead, one must first make a contour deformation.

We define $\gamma_R^\pm=\{z\pm iR/\sqrt{2}:z\in\R\}$ oriented so that the strip $|\Im(\lambda)|<R/\sqrt{2}$ (which, by lemma~\ref{lem:Zeros}, contains all zeros of $\Delta$) lies to the right of each contour $\gamma_R^\pm$. By lemma~\ref{lem:Decay}, and integrating by parts in $x$, as $\lambda\to\infty$ from within $\overline{\{\lambda\in\C^+\setminus D_R^+:\Im(\lambda)>R/\sqrt{2}\}}$,
\BES
	\lambda \re^{i\lambda x}\frac{\zeta^+(\lambda;q_0)}{\Delta(\lambda)} = \mathcal{O}\left(|\lambda|\re^{-|\lambda| x/\sqrt{2}}\right),
\EES
uniformly in $\arg(\lambda)$, and, as $\lambda\to\infty$ from within $\overline{\{\lambda\in\C^-\setminus D_R^-:\Im(\lambda)<-R/\sqrt{2}\}}$,
\BES
	\lambda \re^{i\lambda (x-1)} \frac{\zeta^-(\lambda;q_0)}{\Delta(\lambda)} = \mathcal{O}\left(\re^{-|\lambda| (1-x)/\sqrt{2}}\right)
\EES
uniformly in $\arg(\lambda)$. Hence, by Jordan's lemma with exponential kernel $\re^{-\lambda^2t}$, subsequently applying equation~\eqref{eqn:zeta.cancel}, and evaluating the limit,
\begin{multline} \label{eqn:Existence.Proof:qx2}
	q_x(1,t) =
		\frac{i}{2\pi} \left\{\int_{-\infty}^{\infty} - \int_{\gamma_R^+} \right\} \lambda\re^{i\lambda-\lambda^2t}\hat{q}_0(\lambda)\D\lambda
		- \frac{i}{2\pi} \left\{\int_{\gamma_R^+} + \int_{\gamma_R^-}\right\}\lambda\re^{-\lambda^2t}\frac{\zeta^-(\lambda;q_0)}{\Delta(\lambda)}\D\lambda \\
		- \frac{i}{2\pi} \left\{\int_{\partial D_R^+} + \int_{\partial D_R^-}\right\} \lambda\re^{i\lambda-\lambda^2t}\frac{H(\lambda;g_0,g_1,\tau)}{\Delta(\lambda)}\D\lambda.
\end{multline}
Using Jordan's lemma with exponential kernel $\re^{-\lambda^2t/2}$, the first pair of integrals in equation~\eqref{eqn:Existence.Proof:qx2} cancel.
Applying a change of variables $\lambda\mapsto-\lambda$ in the fourth integral of equation~\eqref{eqn:Existence.Proof:qx2} shows that the third and fourth integrals cancel. The remaining terms in~\eqref{eqn:Existence.Proof:qx2} simplify to
\BE
	q_x(1,t) = \frac{-1}{2\pi}\int_{\partial D_R^+}2i\lambda\re^{-\lambda^2t} \int_0^\tau \re^{\lambda^2s}g_1(s)\D s \D \lambda.
\EE
Applying a change of variables $\lambda^2=-i\rho$, $\lambda=i\sqrt{i\rho}$ (for the principal branch of the square root), we obtain
\BES
	2\pi q_x(1,t) = -\int_{-\R}\re^{i\rho t} \int_0^\tau \re^{-i\rho s}g_1(s)\D s \D \rho.
\EES
Hence, by the usual Fourier inversion theorem, $q$ satisfies the boundary condition~\eqref{eqn:INVP.BC}.

It only remains to show that $q$, as defined by equation~\eqref{eqn:Solution.q}, satisfies the nonlocal condition~\eqref{eqn:INVP.NC} for $t\in(0,t)$, $\tau>t$. We integrate equation~\eqref{eqn:Solution.q} against $K(x)$ from $x=0$ to $1$ and apply uniform convergence to obtain
\begin{multline} \label{eqn:Existence.Proof:Kq0}
	2\pi\int_0^1K(x)q(x,t)\D x =
		\int_{-\infty}^\infty \int_0^1K(x)\re^{i\lambda x}\D x\; \re^{-\lambda^2t}\hat{q}_0(\lambda)\D\lambda \\
		-\left\{ \int_{\partial D_R^+}\frac{\zeta^+(\lambda;q_0)}{\Delta(\lambda)} + \int_{\partial D_R^-}\frac{\re^{-i\lambda}\zeta^-(\lambda;q_0)}{\Delta(\lambda)} \right\} \int_0^1K(x)\re^{i\lambda x}\D x\; \re^{-\lambda^2t} \D\lambda \\
		-\left\{ \int_{\partial D_R^+} + \int_{\partial D_R^-} \right\} \int_0^1K(x)\re^{i\lambda x}\D x\; \re^{-\lambda^2t} \frac{H(\lambda;g_0,g_1,\tau)}{\Delta(\lambda)}\D\lambda.
\end{multline}

By lemma~\ref{lem:Decay}, and applying the Riemann-Lebesgue lemma to the integral, as $\lambda\to\infty$ from within $\overline{\{\lambda\in\C^+\setminus D_R^+:\Im(\lambda)>R/\sqrt{2}\}}$,
$$
	\frac{\zeta^+(\lambda;q_0)}{\Delta(\lambda)} \int_0^1K(x)\re^{i\lambda x}\D x = o(1),
$$
uniformly in $\arg(\lambda)$, and, as $\lambda\to\infty$ from within $\overline{\{\lambda\in\C^-\setminus D_R^-:\Im(\lambda)<-R/\sqrt{2}\}}$,
$$
	\frac{\re^{-i\lambda}\zeta^-(\lambda;q_0)}{\Delta(\lambda)} \int_0^1K(x)\re^{i\lambda x}\D x = o(|\lambda|^{-1}),
$$
uniformly in $\arg(\lambda)$.
Hence, by Jordan's lemma with exponential kernel $\re^{-\lambda^2t}$, the second line of equation~\eqref{eqn:Existence.Proof:Kq0} can be rewritten as
$$
	-\left\{ \int_{\gamma_R^+}\frac{\zeta^+(\lambda;q_0)}{\Delta(\lambda)} + \int_{\gamma_R^-}\frac{\re^{-i\lambda}\zeta^-(\lambda;q_0)}{\Delta(\lambda)} \right\} \int_0^1K(x)\re^{i\lambda x}\D x\; \re^{-\lambda^2t} \D\lambda.
$$
Applying equation~\eqref{eqn:zeta.cancel},
\begin{multline} \label{eqn:Existence.Proof:Kq1}
	2\pi\int_0^1K(x)q(x,t)\D x =
		\left\{ \int_{-\infty}^\infty - \int_{\gamma_R^+} \right\} \int_0^1K(x)\re^{i\lambda x}\D x\; \re^{-\lambda^2t}\hat{q}_0(\lambda)\D\lambda \\
		-\left\{ \int_{\gamma_R^+} + \int_{\gamma_R^-} \right\} \frac{\zeta^-(\lambda;q_0)}{\Delta(\lambda)} \int_0^1K(x)\re^{i\lambda (x-1)}\D x\; \re^{-\lambda^2t} \D\lambda \\
		-\left\{ \int_{\partial D_R^+} + \int_{\partial D_R^-} \right\} \int_0^1K(x)\re^{i\lambda x}\D x\; \re^{-\lambda^2t} \frac{H(\lambda;g_0,g_1,\tau)}{\Delta(\lambda)}\D\lambda.
\end{multline}

As $\lambda\to\infty$ from within $\overline{\{\lambda\in\C^+:\Im(\lambda)<R/\sqrt{2}\}}$,
$$
	\hat{q}_0(\lambda) \int_0^1K(x)\re^{i\lambda x}\D x = o(|\lambda|^{-1}),
$$
uniformly in $\arg(\lambda)$. Hence, by Jordan's lemma with exponential kernel $\re^{-\lambda^2t}$, the first pair of integrals in equation~\eqref{eqn:Existence.Proof:Kq1} cancel.
A change of variables $\lambda\mapsto-\lambda$ in the fourth integral on the right of equation~\eqref{eqn:Existence.Proof:Kq1} shows that the sum of the third and fourth terms on the right of equation~\eqref{eqn:Existence.Proof:Kq1} evaluates to
$$
	-2\int_{\gamma_R^+} \re^{-\lambda^2t} \zeta^-(\lambda;q_0) \D\lambda = -2\int_{-\infty}^\infty \re^{-\lambda^2t} \zeta^-(\lambda;q_0) \D\lambda,
$$
where the latter equality is established by applying Jordan's lemma again with exponential kernel $\re^{-\lambda^2t/2}$. This integral converges and the integrand is odd, so it evaluates to $0$.
Therefore
\begin{multline*}
	2\pi \int_0^1 K(x) q(x,t)\D x = - \int_{\partial D_R^+} \frac{\re^{-\lambda^2t}}{\Delta(\lambda)} \Bigg(
		 \int_0^\tau\re^{\lambda^2s} g_0(s)\D s \; 2i\lambda \int_0^1K(x)\cos([1-x]\lambda)\D x \\
		+ \int_0^\tau\re^{\lambda^2s} g_1(s)\D s \left[ \int_0^1K(x)\re^{i\lambda x}\D x \int_0^1K(y)\re^{-i\lambda y}\D y - \int_0^1K(y)\re^{i\lambda y}\D y \int_0^1K(x)\re^{-i\lambda x}\D x \right]
	\Bigg)
	\D \lambda,
\end{multline*}
and the terms in the bracket cancel.
Hence, by the change of variables $\lambda^2=-i\rho$, $\lambda=i\sqrt{i\rho}$, $q$ satisfies the nonlocal condition~\eqref{eqn:INVP.NC}.

This completes the proof of theorem~\ref{thm:Solution}.

\begin{proof}[Proof of corollary~\ref{cor:RegularityOfSolution}]
	Although the three paragraphs containing equations~\eqref{eqn:Existence.Proof:qx1}--\eqref{eqn:Existence.Proof:qx2} were used to show that the boundary condition is satisfied, they also contain an argument that $q_x(1,t)$, interpreted as $\lim_{x\to1^-}q_x(x,t)$, converges.
	The argument that $q_x(0,t)$, interpreted as $\lim_{x\to0^+}q_x(x,t)$, also converges is very similar.
	That $\partial_x^j\partial_t^kq(x,t)$ converges for interior $x,t$ is immediate from the exponential decay of $\re^{i\la x}$, $\re^{i\la(x-1)}$ as $\la\to\infty$ along $\partial D_R^\pm$ and the exponential decay of $\re^{-\la^2t}$ as $\la\to\infty$ along $\R$.
	Higher order spatial derivatives only insert extra factors of $i\la$ in each integrand, and any monomial in $\la$ is dominated by the exponential $\re^{i\la x}$ (or $\re^{i\la (1-x)}$), for all positive $x$.
	Therefore $q$ is infinitely differentiable in $x$, with each derivative continuously extensible to $[0,1]$.
	Moreover, each function $\partial_x^jq(x,t)$ obtained as described above is differentiable with respect to $t$, and exchanging the partial derivatives yields the same formula, so the map $t\to \partial_x^jq(x,t)$ is continuous in $t$.
	Similarly, the map $x\to \partial_t^kq(x,t)$ is continuous in $x$.
\end{proof}

\begin{rmk}
	To extend time derivatives of $q$ to $t=0$ of course requires additional differentiability of $q_0$.
	However, there is still an additional issue.
	Indeed, $q_0$ may be arbitrarily smooth on $[0,1]$ without it's zero extension to $\R$ having more than a single weak derivative.
	In such a situation, we cannot expect convergence of the inverse Fourier transform
	\BE \label{eqn:tDerivativesAt0}
		\int_{-\infty}^{\infty}\la^{2n}\re^{i\la x}\hat{q}_0(\la)\D\la,
	\EE
	which arises in $\lim_{t\to0^+}\partial_t^nq(x,t)$.
	In initial-boundary value problems studied via the Fokas method, there are two ways to get around this difficulty.
	The first is to interpret $\hat{q}_0$ as the Fourier transform of a compactly supported $C^n$ extension of $q_0$, so that the integral converges.
	An alternative approach~\cite{BT2015a} is to fix $\tau>0$ (rather than allowing $\tau\to0$ with $t$) and interpret the singular integral~\eqref{eqn:tDerivativesAt0} jointly with the remaining singular integrals of $H$ about $\partial D_R^\pm$.
	The latter approach requires strong compatibility of the initial and boundary data, but permits extension of the smoothness results up to the corners $(x,t)=(0,0)$ and $(1,0)$.
	It is expected that similar results will hold for initial-nonlocal value problems, but the analysis remains open.
\end{rmk}

\section{Relation with multipoint problems} \label{sec:Multipoint}

Initial-nonlocal value problem~\eqref{eqn:INVP} is related to initial-multipoint value problems that have already been studied using the Fokas transform method~\cite{PS2015b}. In this section, we describe these relations and use the method of the present work to inform a new approach to the Dirichlet-to-Neumann map for multipoint problems.

\subsection{Nonlocal value problem~\eqref{eqn:INVP} as a limit of multipoint value problems}

We aim to approximate problem~\eqref{eqn:INVP} as an initial-multipoint problem so that the framework of~\cite{PS2015b} may be applied to solve the approximate problem. If $K$ were a piecewise linear function then the results of that paper could be applied directly to transform the nonlocal condition into an equivalent multipoint condition, but for general $K$ this is not possible. Instead, we must devise a sequence of multipoint conditions which has as its limit the nonlocal condition~\eqref{eqn:INVP.NC}. One approach would be to approximate $K$ by a sequence of piecewise linear functions, and construct the equivalent multipoint conditions in each case. However, due to the complexity of those multipoint conditions, the solution representation thus obtained is somewhat unwieldy. We prefer to take the following approach instead.

For $m\in\N$, let
\BE \label{eqn:Km.defn}
	K_m(x) = \frac{1}{m+1}\sum_{j=0}^m\delta\left(x-\frac{j}{m}\right)K\left(\frac{j}{m}\right).
\EE
Then $K_m\xrightarrow{w^\star}K$ with test functions $q$ continuous in $x$. But
\BE
	\int_0^1K_m(x)q(x,t)\D x = \frac{1}{m+1}\sum_{j=0}^m K\left(\frac{j}{m}\right) q\left(\frac{j}{m},t\right),
\EE
so nonlocal condition~\eqref{eqn:INVP.NC} is the weak-$\star$ limit of the $(m+1)$-point condition
\BE \label{eqn:INVP.MC}
	\sum_{j=0}^mb_m^jq\left(\frac{j}{m},t\right) = g_0(t), \qquad t\in[0,T],
\EE
as $m\to\infty$, where
\BE
	b_m^j = \frac{1}{m+1}K\left(\frac{j}{m}\right).
\EE

\subsection{Solution of multipoint value problems}

In this section, we apply the results of~\cite{PS2015b} to solve the initial-multipoint value problem
\begin{subequations} \label{eqn:IMVP}
\begin{align} \label{eqn:IMVP.PDE}
	[\partial_t-\partial_x^2]q(x,t) &= 0 & (x,t) &\in (0,1)\times(0,T), \\ \label{eqn:IMVP.IC}
	q(x,0) &= q_0(x) & x &\in[0,1], \\ \label{eqn:IMVP.BC}
	q_x(1,t) &= g_1(t) & t &\in [0,T], \\ \label{eqn:IMVP.MC}
	\int_0^1 K_m(x)q(x,t) \D x &= g_0(t) & t &\in [0,T],
\end{align}
\end{subequations}
for finite $m$, where $K_m$ is defined by equation~\eqref{eqn:Km.defn}, so~\eqref{eqn:IMVP.MC} is really a multipoint condition rather than a genuine nonlocal condition.

The following proposition is a direct application of~\cite{PS2015b}, with only the cancellation of a factor of $2i\lambda$, to simplify the presentation.

\begin{prop} \label{prop:IMVP.Solution}
	Suppose $m\in\N$ and $q$ satisfies initial-multipoint value problem~\eqref{eqn:INVP.PDE}--\eqref{eqn:INVP.BC},~\eqref{eqn:INVP.MC}. Then there exists $R>0$ sufficiently large that
	\begin{subequations} \label{eqn:Solution.m}
	\begin{multline} \label{eqn:Solution.m.q}
		q(x,t) = \frac{1}{2\pi}\int_{-\infty}^\infty \re^{i\lambda x-\lambda^2t} \hat{q}_0(\lambda) \D\lambda
			- \int_{\partial D_R^+} \re^{i\lambda x-\lambda^2t} \frac{F_m^+(\lambda)-\mathcal{H}_m(\lambda)}{\Delta_m(\lambda)} \D\lambda \\
			- \int_{\partial D_R^-} \re^{i\lambda x-\lambda^2t} \frac{\re^{-i\lambda}F_m^-(\lambda)-\mathcal{H}_m(\lambda)}{\Delta_m(\lambda)} \D\lambda,
	\end{multline}
	where
	\begin{align} \label{eqn:Solution.m.F+}
		F_m^+(\lambda) &= \sum_{j=0}^mb_m^j\cos\left(\left[1-\frac{j}{m}\right]\lambda\right)\int_{0}^{\frac{j}{m}}\re^{-i\lambda z}q_0(z)\D z + \sum_{j=0}^m \re^{-i\lambda\frac{j}{m}}\int_{\frac{j}{m}}^{1}\cos\left(\left[1-z\right]\lambda\right)q_0(z)\D z, \\ \label{eqn:Solution.m.F-}
		F_m^-(\lambda) &= i\sum_{j=0}^mb_m^j\int_{\frac{j}{m}}^{1}\sin\left(\left[z-\frac{j}{m}\right]\lambda\right)q_0(z)\D z, \\ \label{eqn:Solution.m.H}
		\mathcal{H}_m(\lambda) &= i\lambda\re^{-i\lambda} \int_0^\tau \re^{\lambda^2s}g_0(s)\D s + \sum_{j=0}^m\re^{-i\lambda\frac{j}{m}}b_m^j \int_0^\tau \re^{\lambda^2s}g_1(s)\D s, \\ \label{eqn:Solution.m.Delta}
		\Delta_m(\lambda) &= \sum_{j=0}^m b_m^j \cos\left(\left[1-\frac{j}{m}\right]\lambda\right),
	\end{align}
	\end{subequations}
	and $\tau\in[t,T]$.
\end{prop}

\begin{rmk} \label{rmk:LimitProblems}
	Understanding the sums over $j$ as Riemann sums and taking the limit as $m\to\infty$, solution representation~\eqref{eqn:Solution.m} approaches solution representation~\eqref{eqn:Solution}. In light of this, one might ask why the direct proof of theorem~\ref{thm:Solution} was given.

	The limit $m\to\infty$ is taken in a na\"{i}ve sense. Without the a priori proof of uniqueness, the $m\to\infty$ argument does not yield any uniqueness result. Moreover, it is not even obvious that $q$ defined by the limit should satisfy the original problem. Of course, the argument of section~\ref{sec:Exist} could be applied to the definition of $q$ obtained through the limit $m\to\infty$ (it is, after all, the same definition!) to obtain an existence and solution representation result. But then uniqueness must be determined through an alternative method. Moreover, in evaluating the relative merits of the two approaches, it should not be ignored that using the approach of~\cite{PS2015b} to derive the result of proposition~\ref{prop:IMVP.Solution} and then taking the limit consumes time comparable to the time taken to derive the solution representation via the new nonlocal method. We therefore conclude that the new approach is superior.
\end{rmk}

\subsection{An initial-boundary value problem as a limit of initial-nonlocal value problems}

A classical initial-boundary value problem may be considered as the limit of initial-nonlocal problems as $K$ approaches a $\delta$ function at $0$.
The caveats regarding such a limit are similar to those discussed in remark~\ref{rmk:LimitProblems}, but we proceed with such a formal calculation regardless.

Suppose that in problem~\eqref{eqn:INVP},
\BES
	K(x)=K_j(x)=\begin{cases}j & \mbox{if }0\leq x \leq\frac{1}{j}, \\ 0 & \mbox{otherwise,} \end{cases}
\EES
and study the limit $j\to\infty$. The limit of nonlocal condition~\eqref{eqn:INVP.NC} is the boundary condition $q(0,t)=g_0(t)$, so we expect to recover the solution of the classical initial-boundary value problem. Indeed, defining $\zeta^\pm_j$, $H_j$ and $\Delta_j$ to depend upon $K_j$, we find
\begin{subequations}
\begin{align}
	\Delta_j(\lambda) &= \cos(\lambda) - \frac{\lambda}{2j}\sin(\lambda) + \mathcal{O}\left(j^{-2}\right), \\
	H_j(\lambda;g_0,g_1,\tau) &= i\lambda \re^{-i\lambda} \int_0^\tau \re^{\lambda^2s}g_0(s) \D s + \left[ 1-\frac{i\lambda}{j} + \mathcal{O}\left(j^{-2}\right) \right] \int_0^\tau \re^{\lambda^2s}g_1(s) \D s, \\
	\zeta_j^-(\lambda;q_0) &= i\int_0^1q_0(z)\sin(z\lambda)\D z - \frac{i\lambda}{j} \int_0^1 q_0(z) \cos(z\lambda) \D z + \mathcal{O}\left(j^{-2}\right), \\
	\zeta_j^+(\lambda;q_0) &= \left[ 1-\frac{i\lambda}{2j} \right] \int_0^1\cos([1-z]\lambda)q_0(z)\D z - \frac{1}{2j}\cos(\lambda)q_0(0) + \mathcal{O}\left(j^{-2}\right),
\end{align}
\end{subequations}
in which the $\mathcal{O}(1)$ terms match the initial-boundary value problem.

\subsection{Alternative Dirichlet-to-Neumann map for multipoint problems}

The new implementation of the Fokas transform method makes it unnecesary to study nonlocal problems via limits of the multipoint Fokas transform method of~\cite{PS2015b}. However, by taking an alternative approach to the Dirichlet-to-Neumann map, informed by the nonlocal Dirichlet-to-Neumann map derived in section~\ref{sec:Unique} of the present work, it is possible to rederive some of the results of~\cite{PS2015b} through an alternative and more efficient argument. In this section, we illustrate this argument for the specific example of problem~\eqref{eqn:IMVP}.

\subsubsection*{Stage 1}

Stage 1 of the Fokas transform method is independent of the boundary / multipoint / nonlocal conditions, so it proceeds as described in section~\ref{ssec:FTM.FiniteHeat}. We obtain the Ehrenpreis form~\eqref{eqn:Ehrenpreis} for any $R>0$ and the global relation~\eqref{eqn:GR} for any $\tau\in[0,T]$ and any $y,z$ satisfying $0 \leq y \leq z \leq 1$.

\subsubsection*{Stage 2}

As in the nonlocal problem, we must use the boundary condition~\eqref{eqn:IMVP.BC} and multipoint condition~\eqref{eqn:IMVP.MC} to find expressions for the spectral functions
\begin{align*}
	&f_0(\lambda;0), & &f_1(\lambda;0), & &f_0(\lambda;1), & &f_1(\lambda;1).
\end{align*}
The boundary condition~\eqref{eqn:IMVP.BC} implies that
\BE \label{eqn:Multipoint.BC.transformed}
	f_1(\lambda;1)=h_1(\lambda),
\EE
for $h_1$ defined by equation~\eqref{eqn:h1.defn} in terms of the datum $g_1$, and
\BE \label{eqn:Multipoint.MC.transformed}
	\int_0^1 K_m(y)f_0(\lambda;y) \D y = i\lambda \int_0^t \re^{\lambda^2s} g_0(s) \D s =: h_0(\lambda).
\EE

In~\cite{PS2015b}, Sheils' implementation of the Dirichlet-to-Neumann map for interface problems~\cite{DPS2014a} was adapted to construct a Dirichlet-to-Neumann map for multipoint problems. Specifically, a global relation was derived for each interval $[j/m,(j+1)/m]$, by setting $y=j/m$, $z=(j+1)/m$ in equation~\eqref{eqn:GR}, for each $j=0,1,\ldots,m-1$. This provided $m$ equations in the $2m+2$ spectral functions $f_k(\lambda;j/m)$ for $j=0,1,\ldots,m$ and $k=0,1$. As usual, the spectral functions have very simple symmetry, allowing the map $\lambda\mapsto-\lambda$ to produce a further $m$ linearly independent equations in the same spectral functions. The $t$-transformed boundary condition~\eqref{eqn:Multipoint.BC.transformed} and multipoint condition~\eqref{eqn:Multipoint.MC.transformed} complete a linear system of rank $2m+2$. Some work was done to rearrange the system so that it may be solved directly and relatively easily for the expressions that appear in the Ehrenpreis form~\eqref{eqn:Ehrenpreis},
$$
	\left[f_0(\lambda;0) + f_1(\lambda;0)\right]
	\qquad\qquad\mbox{and}\qquad\qquad
	\re^{-i\lambda}\left[f_0(\lambda;1) + f_1(\lambda;1)\right],
$$
without solving the full system. However, it is still necessary to solve for two entries in a rank-$(m+2)$ system.
The linear system was solved for general $m$ in~\cite[lemma~5.1]{PS2015b}.

In place of the $m$ global relations, we prefer to use two global relation equations: the global relation of two-point type~\eqref{eqn:GR.usual} and the \emph{global relation of multipoint type}
\begin{multline} \label{eqn:GR.multipoint}
	\int_0^1 K_m(y) f_1(\lambda;y) \D y - f_0(\lambda;1) \re^{-i\lambda} \int_0^1 \re^{i\lambda y} K_m(y) \D y =
	-h_0(\lambda) + h_1(\lambda)\re^{-i\lambda}\int_0^1 \re^{i\lambda y} K_m(y) \D y +{} \\
	\int_0^1 \re^{i\lambda y} K_m(y)\hat{q}_0(\lambda;y,1)\D y - \re^{\lambda^2 t}\int_0^1 \re^{i\lambda y} K_m(y)\hat{q}(\lambda,t;y,1)\D y,
\end{multline}
which was obtained from equation~\eqref{eqn:GR} by evaluating at $\tau=t$, $z=1$, multiplying by $\re^{i\lambda y}K_m(y)$ and integrating in $y$ from $0$ to $1$. This is exactly the same as the global relation of nonlocal type~\eqref{eqn:GR.nonlocal}, except that, because of the definition of $K_m$ as a sum of $\delta$-functions, each integral in~\eqref{eqn:GR.multipoint} represents a finite sum. Under the map $\lambda\mapsto-\lambda$, we obtain a further two linear equations in the four unknown spectral functions and the spectral integrals
\BES
	\int_0^1 K_m(y) f_0(\lambda;y) \D y \qquad\qquad\mbox{and}\qquad\qquad \int_0^1 K_m(y) f_1(\lambda;y) \D y.
\EES
Together with the transformed boundary condition~\eqref{eqn:Multipoint.BC.transformed} and mutlipoint condition~\eqref{eqn:Multipoint.MC.transformed}, this specifies a linear system of rank 6, uniformly in $m$. Moreover, as this is the same system as was studied in section~\ref{ssec:NonlocalDtoNMap}, it can be solved in the same way.

\begin{rmk} \label{rmk:MultipointSmallerDtoNMap}
	In~\cite{PS2015b}, the general second and third order $(m+1)$-point Dirichlet-to-Neumann maps were solved explicitly. Although, for higher $n$\Th spatial order problems, the multipoint Dirichlet-to-Neumann map was set up as a rank $n(m+1)$ linear system, the system was not solved explicitly for $n>3$; it is not difficult to see how the approach for $n=3$ can be generalised, but the notation necessarily becomes cumbersome. A generalisation of the new approach (to admit partial differential equations~\eqref{eqn:GeneralPDE} of spatial order $n$, together with $n$ multipoint conditions, instead of just one and a boundary condition) yields a linear system of rank $n(n+1)$, uniformly in $m$. This may be more tractible than the system whose rank grows linearly with $m$, at least for $n\ll m$.
\end{rmk}

\section{General nonlocal problems for the heat equation} \label{sec:General}

Consider the initial-nonlocal value problem
\begin{subequations} \label{eqn:GINVP}
\begin{align} \label{eqn:GINVP.PDE}
	[\partial_t-\partial_x^2]q(x,t) &= 0 & (x,t) &\in (0,1)\times(0,T), \\ \label{eqn:GINVP.IC}
	q(x,0) &= q_0(x) & x &\in[0,1], \\ \label{eqn:GINVP.NC0}
	\int_0^1 K(x)q(x,t) \D x &= g_0(t) & t &\in [0,T], \\ \label{eqn:GINVP.NC1}
	\int_0^1 L(x)q(x,t) \D x &= g_1(t) & t &\in [0,T],
\end{align}
\end{subequations}
where the data $g_0,g_1,q_0$ are sufficiently smooth.

Provided problem~\eqref{eqn:GINVP} is well-posed, the Fokas transform method may provide a means of solution. Stage~1 is independent of the nonlocal conditions, so it proceeds as described in section~\ref{ssec:FTM.FiniteHeat}. In stage~2, we must construct a Dirichlet-to-Neumann map in the form of a linear system that may be solved for the sums of spectral functions
\begin{align*}
	&f_0(\lambda;0), & &f_1(\lambda;0), & &f_0(\lambda;1), & &f_1(\lambda;1).
\end{align*}
We use three evaluations of the global relation~\eqref{eqn:GR}. One evaluation, $y=0$ $z=1$, yields the usual global relation of two-point form~\eqref{eqn:GR.usual}. We construct two global relations of nonlocal type by evaluating equation~\eqref{eqn:GR} at $y=0$, multiplying by either $\re^{i\lambda z} K(z)$ or $\re^{i\lambda z} L(z)$, and integrating from $0$ to $1$:
\begin{multline*}
	\left[ f_0(\lambda;0)+f_1(\lambda;0) \right] \int_0^1 \re^{i\lambda z}K(z)\D z - \int_0^1 K(z)f_1(\lambda;z)\D z \\
	= h_0(\lambda) + \int_0^1 \re^{i\lambda z} K(z)\hat{q}_0(\lambda;0,z)\D z - \re^{\lambda^2t}\int_0^1 \re^{i\lambda z} K(z)\hat{q}(\lambda,t;0,z)\D z,
\end{multline*}
and
\begin{multline*}
	\left[ f_0(\lambda;0)+f_1(\lambda;0) \right] \int_0^1 \re^{i\lambda z}L(z)\D z - \int_0^1 L(z)f_1(\lambda;z)\D z \\
	= h_1(\lambda) + \int_0^1 \re^{i\lambda z} L(z)\hat{q}_0(\lambda;0,z)\D z - \re^{\lambda^2t}\int_0^1 \re^{i\lambda z} L(z)\hat{q}(\lambda,t;0,z)\D z,
\end{multline*}
where
\BES
	h_j(\lambda) := i\lambda\int_0^t\re^{-\lambda^2s}g_j(s)\D s
\EES
are data.
Provided, for some $x$,
\BE \label{eqn:General.FullRankCondition}
	\int_0^1 K(1-z)L(x-z)\D z \neq \int_0^1 L(1-z)K(x-z)\D z,
\EE
applying the maps $\lambda\mapsto\lambda$ and $\lambda\mapsto-\lambda$ yields a full rank system of four equations, so we can solve it in particular for $[f_0(\lambda;0)+f_1(\lambda;0)]$. Applying the global relation of two-point type yields an expression for $\re^{-i\lambda}[f_0(\lambda;1)+f_1(\lambda;1)]$.

Note that problem~\eqref{eqn:GINVP} includes as special cases not only all initial-nonlocal value problems for the heat equation, but also all initial-multipoint value problems, hence all initial-boundary value problems for the heat equation. Indeed, for an arbitrary partition
\BES
	0=\eta_0<\eta_1<\ldots<\eta_m=1,
\EES
we may compose $K$ and $L$ from parts
\begin{align*}
	K &= K_0 + K_1 + K_c, & L &= L_0 + L_1 + L_c, \\
\intertext{where}
	K_0(x) &= \sum_{j=0}^m \M{k}{0}{j}\delta\left(x-\eta_j\right), & L_0(x) &= \sum_{j=0}^m \M{\ell}{0}{j}\delta\left(x-\eta_j\right), \\
	K_1(x) &= \sum_{j=0}^m \M{k}{1}{j}\delta'\left(x-\eta_j\right), & L_1(x) &= \sum_{j=0}^m \M{\ell}{1}{j}\delta'\left(x-\eta_j\right),
\end{align*}
and the remainders $K_c$, $L_c$ are sufficiently smooth that lemmata~\ref{lem:Zeros} and~\ref{lem:Decay} may be appropriately adapted.

\appendix

\section{Proofs of lemmata~\ref{lem:Zeros} and~\ref{lem:Decay}} \label{sec:AppProofs}

The first proof follows the presentation by Langer~\cite[theorem~14]{Lan1931a}, though Langer recognises Cartwright's earlier attribution of the result to Hardy.

\begin{proof}[Proof of lemma~\ref{lem:Zeros}]
	We define
	\BES
		\psi(y) = \begin{cases} \frac{1}{2}k(y) & 0\leq y\leq1, \\ \frac{1}{2}k(-y) & -1\leq y<0, \end{cases}
	\EES
	so that
	\begin{align*}
		\Delta(\lambda) &= \int_{-1}^1\psi(y)\re^{i\lambda y} \D y, \\
		V_{b-\delta}^{b}(\psi) &= \max\{|\psi(y_2)-\psi(y_1)|:b-\delta \leq y_1<y_2 \leq b\} = V_{b-\delta}^{b}(k).
	\end{align*}

	The function $\psi$ is continuous at the endpoints of its support $\pm b$ and is of bounded variation, with $V_{-1}^{1}(\psi)=V_{0}^{1}(k)$.
	Therefore
	\BE \label{eqn:zeroslem.est0}
		i \lambda \Delta(\lambda) = \psi(b) \re^{i\lambda b} - \psi(-b) \re^{-i\lambda b} - \int_{-b}^{b-\delta} \re^{i\lambda y} \D\psi(y) - \int_{b-\delta}^b \re^{i\lambda y} \D\psi(y).
	\EE
	Suppose $\lambda=\zeta-i\theta$ for some $\theta>0$.
	Then
	\begin{align*}
		\left| \int_{-b}^{b-\delta} \re^{i\lambda y} \D\psi(y) \right| &\leq V_{-1}^{1}(\psi) \re^{\theta(b-\delta)}, \\
		\left| \int_{b-\delta}^b \re^{i\lambda y} \D\psi(y) \right| &\leq V_{b-\delta}^{b}(\psi) \re^{\theta}.
	\end{align*}
	Because $k$ is continuous at $1$, there exists some sufficiently small $\delta_0$ satisfying the definition in the statement of the lemma.
	In particular, for $\delta=\delta_0$, and $\theta>M$ we obtain
	\begin{align} \label{eqn:zeroslem.est1}
		\left| \int_{-b}^{b-\delta_0} \re^{i\lambda y} \D\psi(y) \right| &\leq \frac{|\psi(b)|}{4} \re^{\theta}, \\ \label{eqn:zeroslem.est2}
		\left| \int_{b-\delta_0}^b \re^{i\lambda y} \D\psi(y) \right| &< \frac{|\psi(b)|}{4} \re^{\theta}, \\ \label{eqn:zeroslem.est3}
		\re^{-2b\theta} &\leq \frac{1}{4}.
	\end{align}
	Combining estimates~\eqref{eqn:zeroslem.est1}--\eqref{eqn:zeroslem.est3}, and observing $\psi(-b)=\psi(b)$, we obtain from equation~\eqref{eqn:zeroslem.est0}
	\BE \label{eqn:EstimateOfModLaDeltaOutsideStrip}
		|\lambda \Delta(\lambda)| > |\psi(b)|\re^{b\theta}\left( 1 - \frac{1}{4} - \frac{1}{4} - \frac{1}{4} \right).
	\EE

	We have shown that $\Delta$ has no zeros with imaginary part less than $-M$.
	The argument for imaginary part greater than $M$ is analogous.
	The factor $\sqrt{2}$ in equation $R=\sqrt{2}M$ appears due to the angle, $\pi/4$, which the ray components of $\partial D_R^\pm$ make with the horizontal strip $|\Im(\lambda)|<M$.
\end{proof}

\begin{proof}[Proof of lemma~\ref{lem:Decay}]
	Despite not being expressed in exactly this way, the argument in the proof of lemma~\ref{lem:Zeros} up to estimate~\eqref{eqn:EstimateOfModLaDeltaOutsideStrip} may be understood as a proof that
	\BE \label{eqn:Delta.asymp}
		\frac{1}{\Delta(\la)} = \mathcal{O}\left(\lvert\la\rvert\re^{-b\lvert\Im(\la)\rvert}\right) \qquad\mbox{as } \Im(\la)\to\pm\infty.
	\EE
	Therefore the same estimate holds as $\la\to\infty$ from within $\overline{D^\pm}$.

	We integrate by parts in the inner integral in the definition of $\zeta^-$:
	\begin{align*}
		\left\lvert \zeta^-(\la;\phi)\right\rvert &= \left\lvert \frac{1}{\la} \int_0^1K(y) \left( - \cos[(1-y)\la] \phi(1) + \phi(y) + \int_y^1 \cos[(z-y)\la]\phi'(z) \D z \right) \D y \right\rvert \\
		&\leq A(\la) + B(\la) + C(\la),
	\end{align*}
	where
	\begin{align*}
		A(\la) &= \left\lvert \frac{\phi(1)}{\la} \int_0^1 K(y)\cos[(1-y)\la] \D y \right\rvert, \\
		B(\la) &= \left\lvert \frac{1}{\la} \int_0^1 K(y)\phi(y) \D y \right\rvert, \\
		C(\la) &= \left\lvert \frac{1}{\la} \int_0^1 K(y) \int_y^1 \cos[(z-y)\la]\phi'(z) \D z \D y \right\rvert.
	\end{align*}
	We will show that each of the quantities $A(\la),B(\la),C(\la)$ are dominated by $\Delta(\la)$
	\begin{align*}
		\frac{A(\la)}{\lvert\Delta(\la)\rvert} &= \frac{\lvert \phi(1)\rvert}{\lvert\la\rvert} \left\lvert \frac{\Delta}{\Delta} \right\rvert = \mathcal{O}\left(\lvert\la\rvert^{-1}\right), \\
		\frac{B(\la)}{\lvert\Delta(\la)\rvert} &= \mathcal{O}\left(\re^{-b\lvert\Im(\la)\rvert}\right) = \mathcal{O}\left(\lvert\la\rvert^{-1}\right).
	\end{align*}
	Using the supremum norm of $\phi'$, we bound the magnitude of the inner integral of $C(\la)$:
	\begin{align*}
		\left\lvert \int_y^1 \cos[(z-y)\la]\phi'(z) \D z \right\rvert &\leq \frac{1}{2} \int_{y}^1 \left[ \re^{\lvert\Im(\la)\rvert(z-y)} + \re^{-\lvert\Im(\la)\rvert(z-y)} \right] \left\lvert \phi'(z) \right\rvert \D z \\
		&\leq \frac{\lVert\phi'\rVert_\infty}{2\lvert\Im(\la)\rvert} \left( \re^{\lvert\Im(\la)\rvert(1-y)} - \re^{-\lvert\Im(\la)\rvert(1-y)} \right) \\
		&\leq \frac{\lVert\phi'\rVert_\infty \re^{\lvert\Im(\la)\rvert(1-y)}}{\lvert\Im(\la)\rvert}.
	\end{align*}
	We define the function $J:[0,2]\to\R$ to be the zero extension of the absolute value of $K$.
	Then $J$ is also a function of bounded variation.
	It follows that
	\begin{align*}
		\frac{C(\la)}{\lvert\Delta(\la)\rvert} &= \frac{\lVert\phi'\rVert_\infty \re^{\lvert\Im(\la)\rvert}}{\lvert\la\Im(\la)\rvert} \left\lvert \int_0^1 K(y) \re^{-\lvert\Im(\la)\rvert y} \D y \right\rvert \mathcal{O}\left(\lvert\la\rvert\re^{-b\lvert\Im(\la)\rvert}\right) \\
		&= \int_{1-b}^2 J(y) \re^{-\lvert\Im(\la)\rvert y} \D y \;\mathcal{O}\left(\lvert\la\rvert^{-1}\re^{(1-b)\lvert\Im(\la)\rvert}\right) \\
		&= \left( \frac{\lvert K(1-b) \rvert \re^{(1-b)\lvert\Im(\la)\rvert}}{\lvert\Im(\la)\rvert} + \frac{1}{\lvert\Im(\la)\rvert} \left\lvert \int_{1-b}^2 \re^{-\lvert\Im(\la)\rvert y} \D J(y) \right\rvert \right) \mathcal{O}\left(\lvert\la\rvert^{-1}\re^{(1-b)\lvert\Im(\la)\rvert}\right) \\
		&= \mathcal{O}\left(\lvert\la\rvert^{-2}\right)
	\end{align*}
	Hence asymptotic formula~\eqref{eqn:DecayLem.-} holds.

	We follow a similar argument to obtain the asymptotic formula for the ratio $\zeta^+(\la)/\Delta(\la)$.
	Integrating by parts in the inner integrals of the definition of $\zeta^+$,
	\BES
		\left\lvert \zeta^+(\la;\phi) \right\rvert \leq E + F + G + P + Q,
	\EES
	where
	\begin{align*}
		E &= \left\lvert \frac{1}{\la} \int_0^1 K(y) \cos([1-y]\la)\re^{-i\la y} \phi(y)\D y \right\rvert, \\
		F &= \left\lvert \frac{\phi(0)}{\la} \int_0^1 K(y)\cos[(1-y)\la] \D y \right\rvert, \\
		G &= \left\lvert \frac{1}{\la} \int_0^1 K(y)\cos[(1-y)\la] \int_0^y \re^{-i\la z}\phi'(z) \D z \D y \right\rvert, \\
		P &= \left\lvert \frac{1}{\la} \int_0^1 K(y) \sin([1-y]\la)\re^{-i\la y} \phi(y)\D y \right\rvert, \\
		Q &= \left\lvert \frac{1}{\la} \int_0^1 K(y)\re^{-i\la y} \int_y^1 \sin[(1-z)\la]\phi'(z) \D z \D y \right\rvert.
	\end{align*}
	Certainly $F(\la)/\lvert\Delta(\la)\rvert=\mathcal{O}(\lvert\la\rvert^{-1})$, by the same argument as for $A(\la)/\lvert\Delta(\la)\rvert$ above.
	We will establish bounds for each of the other terms.
	\BES
		E
		\leq
		\frac{1}{\lvert \la \rvert} \left\lvert \int_0^1 K(y)\phi(y) \frac{1}{2}\left( \re^{i\la(1-2y)} + \re^{-i\la} \right) \D y \right\rvert
		\leq
		\frac{\lVert K\phi \rVert_1}{2\lvert \la \rvert} \sup_{1-b \leq y \leq 1-a} \left\lvert \re^{i\la(1-2y)} + \re^{-i\la} \right\rvert
		\leq
		\frac{\lVert K\phi \rVert_1}{\lvert \la \rvert} \re^{\Im(\la)}.
	\EES
	Therefore, as $\la\to\infty$ from within $\overline{D^+}$,
	\BES
		\frac{E(\la)}{\lvert\Delta(\la)\rvert}=\mathcal{O}(\re^{\Im(\la)(1-b)}),
	\EES
	which is $\mathcal{O}(1)$ provided $b=1$.
	A similar argument yields $P(\la)/\lvert\Delta(\la)\rvert=\mathcal{O}(\re^{\Im(\la)(1-b)})$.
	Using the supremum norm of the inner integral of $G$, we bound the magnitude of the inner integral of $G(\la)$:
	\BES
		\left\lvert \int_0^y \re^{-i\la z}\phi'(z)\D z \right\rvert \leq \lVert \phi' \rVert_\infty \int_0^y \left\lvert \re^{-i\la z} \right\rvert \D z = \frac{\lVert \phi'\rVert_\infty}{\Im(\la)} \re^{\Im(\la)}.
	\EES
	It follows that
	\begin{align*}
		\lvert G(\la) \rvert &\leq \frac{\lVert\phi'\rVert_\infty}{2\Im(\la)\lvert\la\rvert} \int_{1-b}^2 J(y) \left[ \re^{\Im(\la)} + \re^{\Im(\la)(2y-1)} \right] \D y \\
		&\leq \frac{\lVert\phi'\rVert_\infty}{2\Im(\la)\lvert\la\rvert} \left( \re^{\Im(\la)}\lVert K \rVert_\infty + \frac{\re^{-\Im(\la)}}{2\Im(\la)} \left[ \lvert K(1-b) \rvert \re^{2\Im(\la)(1-b)} + \int_{1-b}^2 \re^{2\Im(\la)y} \D J(y) \right] \right) \\
		&= \frac{\lVert\phi'\rVert_\infty}{2\Im(\la)\lvert\la\rvert} \left( \re^{\Im(\la)}\lVert K \rVert_\infty + \mathcal{O}(\lvert\la\rvert^{-1}) \right).
	\end{align*}
	Therefore, provided $b=1$, we have $G(\la)/\lvert\Delta(\la)\rvert = \mathcal{O}(\lvert\la\rvert^{-1})$.
	The argument for $Q(\la)/\lvert\Delta(\la)\rvert = \mathcal{O}(\lvert\la\rvert^{-1})$ is very similar.
	We have established asymptotic formula~\eqref{eqn:DecayLem.+}.
\end{proof}

\section*{Acknowledgments}
Peter D. Miller acknowledges the support of the National Science Foundation under grant DMS-1513054.
The authors wish to thank David Stewart for providing a Chemist's perspective on the proposed experimental method.

\bibliographystyle{amsplain}
\bibliography{dbrefs}

\providecommand{\bysame}{\leavevmode\hbox to3em{\hrulefill}\thinspace}
\providecommand{\MR}{\relax\ifhmode\unskip\space\fi MR }
\providecommand{\MRhref}[2]{%
  \href{http://www.ams.org/mathscinet-getitem?mr=#1}{#2}
}
\providecommand{\href}[2]{#2}
\begin{thebibliography}{10}

\bibitem{APSS2015a}
M.~Asvestas, E.~P. Papadopoulou, A.~G. Sifalakis, and Y.~G. Saridakis,
  \emph{The unified transform for a class of reaction-diffusion problems with
  discontinuous time dependent parameters}, Proceedings of the World Congress
  on Engineering, vol.~1, 2015.

\bibitem{Ber1999a}
Fredrik Berntsson, \emph{A spectral method for solving the sideways heat
  equation}, Inverse Problems \textbf{15} (1999), no.~4, 891.

\bibitem{BT2015a}
G.~Biondini and T.~Trogdon, \emph{Boundary value problems for evolution partial
  differential equations with discontinuous data}, arXiv:1510.02033 [math.AP],
  2015.

\bibitem{Can1963a}
J.~R. Cannon, \emph{The solution of the heat equation subject to the
  specification of energy}, Quart. Appl. Math \textbf{21} (1963), 155--160.

\bibitem{CEH1987a}
J.~R. Cannon, S.~P. Esteva, and J.~van~der Hoek, \emph{A {G}alerkin procedure
  for the diffusion equation subject to the specification of mass}, SIAM
  Journal on Numerical Analysis \textbf{24} (1987), no.~3, 499--515.

\bibitem{CLM1993a}
J.~R. Cannon, Y.~Lin, and A.~L. Matheson, \emph{The solution of the diffusion
  equation in two spatial variables subject to the specification of mass},
  Appl. Anal. \textbf{50} (1993), 1--15.

\bibitem{DM1963a}
K.~L. Deckert and C.~G. Maple, \emph{Solutions for diffusion equations with
  integral type boundary conditions}, Proc. Iowa Acad. Sci. \textbf{70} (1963),
  354--361.

\bibitem{DPS2014a}
B.~Deconinck, B.~Pelloni, and N.~Sheils, \emph{Non-steady-state heat conduction
  in composite walls}, Proc. R. Soc. Lond. Ser. A Math. Phys. Eng. Sci.
  \textbf{470} (2014), no.~2165, 20130605.

\bibitem{DS2014a}
B.~Deconinck and N.~Sheils, \emph{Interface problems for dispersive equations},
  Stud. Appl. Math. \textbf{134} (2015), 253--275.

\bibitem{DS2016a}
B.~Deconinck and N.~E. Sheils, \emph{Initial-to-interface maps for the heat
  equation on composite domains}, Stud. Appl. Math. \textbf{137} (2016),
  140--154.

\bibitem{DSS2015a}
B.~Deconinck, N.~E. Sheils, and D.~A. Smith, \emph{The linear {K}d{V} equation
  with an interface}, Comm. Math. Phys. \textbf{347} (2016), 489--509.

\bibitem{DTV2014a}
B.~Deconinck, T.~Trogdon, and V.~Vasan, \emph{The method of {F}okas for solving
  linear partial differential equations}, SIAM Rev. \textbf{56} (2014), no.~1,
  159--186.

\bibitem{Fok2008a}
A.~S. Fokas, \emph{A unified approach to boundary value problems}, CBMS-SIAM,
  2008.

\bibitem{FP2001a}
A.~S. Fokas and B.~Pelloni, \emph{Two-point boundary value problems for linear
  evolution equations}, Math. Proc. Cambridge Philos. Soc. \textbf{131} (2001),
  521--543.

\bibitem{FS2016a}
A.~S. Fokas and D.~A. Smith, \emph{Evolution {P}{D}{E}s and augmented
  eigenfunctions. {F}inite interval}, Adv. Differential Equations.

\bibitem{Lan1931a}
R.~E. Langer, \emph{The zeros of exponential sums and integrals}, Bull. Amer.
  Math. Soc. \textbf{37} (1931), 213--239.

\bibitem{Pel2004a}
B.~Pelloni, \emph{Well-posed boundary value problems for linear evolution
  equations on a finite interval}, Math. Proc. Cambridge Philos. Soc.
  \textbf{136} (2004), 361--382.

\bibitem{Pel2005a}
\bysame, \emph{The spectral representation of two-point boundary-value problems
  for third-order linear evolution partial differential equations}, Proc. R.
  Soc. Lond. Ser. A Math. Phys. Eng. Sci. \textbf{461} (2005), 2965--2984.

\bibitem{PS2013a}
B.~Pelloni and D.~A. Smith, \emph{Spectral theory of some non-selfadjoint
  linear differential operators}, Proc. R. Soc. Lond. Ser. A Math. Phys. Eng.
  Sci. \textbf{469} (2013), no.~2154, 20130019.

\bibitem{PS2015b}
B.~Pelloni and D.~A. Smith, \emph{Nonlocal and multipoint value problems for
  linear evolution equations}, arXiv:1511.07244 [math.AP], 2015.

\bibitem{PS2016a}
B.~Pelloni and D.~A. Smith, \emph{Evolution {P}{D}{E}s and augmented
  eigenfunctions. {H}alf line}, J. Spectr. Theory \textbf{6} (2016), 185--213.

\bibitem{DS2015a}
N.~E. Sheils and B.~Deconinck, \emph{Interface problems for dispersive
  equations}, Studies in Applied Mathematics \textbf{134} (2015), no.~3,
  253--275.

\bibitem{SS2015a}
N.~E. Sheils and D.~A. Smith, \emph{Heat equation on a network using the
  {F}okas method}, Journal of Physics A: Mathematical and Theoretical
  \textbf{48} (2015), no.~33, 21 pp.

\bibitem{Smi2012a}
D.~A. Smith, \emph{Well-posed two-point initial-boundary value problems with
  arbitrary boundary conditions}, Math. Proc. Cambridge Philos. Soc.
  \textbf{152} (2012), 473--496.

\bibitem{Smi2013a}
\bysame, \emph{Well-posedness and conditioning of 3rd and higher order
  two-point initial-boundary value problems}, 2013.

\bibitem{Smi2015a}
\bysame, \emph{The unified transform method for linear initial-boundary value
  problems: a spectral interpretation}, Unified transform method for boundary
  value problems: applications and advances, SIAM, Philadelphia, PA, 2015.

\bibitem{Smi2017a}
D.~A. Smith, \emph{Gitlab repository: nonlocal-value-problems-code},
  \url{https://gitlab.com/dazsmith/nonlocal-value-problems-code/blob/a964f7aa7d4732b9a89b006c4f5cb0f636967c92/Nonlocal-heat.ipynb},
  2017.

\end{thebibliography}

\end{document}